\documentclass[11pt]{article}

\usepackage{amsmath}
\usepackage{amsfonts}
\usepackage{color}
\usepackage{amssymb}
\usepackage{graphicx}
\usepackage{enumerate}
\usepackage[all]{xy}
\usepackage{multirow}
\usepackage{array}
\usepackage{tabularx}

\def\P{{\rm \textup{P}}}
\def\S{{\rm \textup{S}}}
\def\L{{\rm \textup{L}}}
\def\G{{\rm \textup{G}}}
\def\A{{\rm \textup{A}}}
\def\M{{\rm \textup{M}}}
\def\U{{\rm \textup{U}}}
\def\Sp{{\rm \textup{Sp}}}
\def\Sz{{\rm \textup{Sz}}}
\def\Re{{\rm \textup{Re}}}
\def\HS{{\rm \textup{HS}}}
\def\Co{{\rm \textup{Co}}}
\def\GF{{\rm \textup{GF}}}
\def\Aut{{\rm \textup{Aut}}}

\def\mod{{\rm mod\,}} 

\def\Out{{\rm Out}}

 \allowdisplaybreaks

\begin{document}

\newtheorem{problem}{Problem}
\newtheorem{theorem}{Theorem}[section]
\newtheorem{corollary}[theorem]{Corollary}
\newtheorem{definition}[theorem]{Definition}
\newtheorem{conjecture}[theorem]{Conjecture}
\newtheorem{question}[theorem]{Question}
\newtheorem{lemma}[theorem]{Lemma}
\newtheorem{proposition}[theorem]{Proposition}
\newtheorem{quest}[theorem]{Question}
\newtheorem{example}[theorem]{Example}
\newenvironment{proof}{\noindent {\bf
Proof.}}{\rule{2mm}{2mm}\par\medskip}
\newenvironment{proofof}{\noindent {\bf
Proof of the Theorem 6.1.}}{\rule{2mm}{2mm}\par\medskip}
\newcommand{\remark}{\medskip\par\noindent {\bf Remark.~~}}
\newcommand{\pp}{{\it p.}}
\newcommand{\de}{\em}

\title{{Block-Transitive $5$-$(v,k,2)$ Designs with $k$ divides $v$}
\thanks{This research was supported  by National Natural Science Foundation of China (No.12401020) and Hunan Provincial Natural Science Foundation (No. 2026JJ60324).
E-mail addresses: huhuijiao1@163.com (H.J. Hu), huangzhengmath@163.com (Z. Huang);
shenshouqiang@126.com (S.Q. Shen, corresponding author).}}

\author{Huijiao Hu$^{1}$, Zheng Huang$^{1}$, Shouqiang Shen$^{2}$  \\[2ex]
{\small $1.$ Department of Mathematics and Statistics, Hunan University of Science and Technology}\\
{\small Xiangtan, Hunan, 411201, P.R. China. }\\
{\small $2.$School of Applied Science, Beijing Information Science and Technology University}\\
{\small Beijing, 100192, P.R. China. }\\}

\maketitle

\vspace{-0.5cm}

\begin{abstract}
The classification of block-transitive 5-designs remains an open problem.
The additional parameter condition that $k$ divides $v$ is called the Camina-Gagen condition.
In this paper, we investigate block-transitive simple $5$-$(v,k,2)$ designs satisfying the Camina-Gagen condition.
Using the classification of finite $2$-homogeneous permutation groups, we consider the affine and almost simple cases separately.
We prove that no such design admits a block-transitive automorphism group of affine type.
For the almost simple case, up to isomorphism, there are exactly two possibilities: a $5$-$(12,6,2)$ design admitting $\P\G\L(2,11)$ as a block-transitive automorphism group and a $5$-$(24,8,2)$ design admitting $\P\G\L(2,23)$ as a block-transitive automorphism group.
\end{abstract}


\vspace{0.5cm}


\section{Introduction}

A $t$-$(v,k,\lambda)$ {\it design} is a finite incidence structure $\mathcal{D}=(\mathcal{P},\mathcal{B})$ with parameters $(v,b,k,\lambda)$, where $\mathcal{P}$ denotes a set of $v$ points, $\mathcal{B}$ denotes a set of $b$ blocks such that every $B\in\mathcal{B}$ is a $k$-subset of $\mathcal{P}$ and any $t$ points are incident with exactly $\lambda$ blocks.
A {\it flag} of $\mathcal{D}$ is a point-block pair $(x,B)$ where $x$ is a point and $B$ is a block with $x\in B$.
Throughout this paper, all designs are assumed to be simple and non-trivial, namely, without repeated blocks and satisfying $5<k<v$.
An {\it automorphism} of $\mathcal{D}$ is a permutation of the points which preserves the blocks.
The set of all automorphisms of $\mathcal{D}$ under the composition of permutations forms a group, denoted by $\Aut(\mathcal{D})$.
If $G\leq\Aut(\mathcal{D})$, then $G$ is called an {\it automorphism group} of $\mathcal{D}$.
$\mathcal{D}$ is {\it flag-transitive} (resp. {\it block-transitive}, {\it point $s$-transitive}, {\it point $s$-homogeneous}) if $G$ acts transitively on the flags (resp. transitively on the blocks, $s$-transitively on points, $s$-homogeneously on points), and we also say that $G$ is {\it flag-transitive} (resp. {\it block-transitive}, {\it point $s$-transitive}, {\it point $s$-homogeneous}).

Designs admitting block-transitive or flag-transitive automorphism groups have been studied extensively.
In 1976, Clapham \cite{Clapham} characterized Steiner triple systems admitting a block-transitive automorphism group
whose action on points is not $2$-transitive.
In 1989, Delandtsheer and Doyen \cite{Delandtsheer} established
an upper bound, in terms of the block size, on the number of points of designs admitting a block-transitive, point-imprimitive automorphism group.
In 1996, Camina \cite{Camina3} reduced the problem of line-transitive point-primitive linear spaces to two cases: those with an elementary abelian group as the socle and those with a non-abelian simple group as the socle.
Subsequently, in 2001, Camina and Praeger \cite{Camina4} generalized this result to line-transitive point-quasiprimitive linear spaces.
In the same year, Li and Liu \cite{Lihl2} presented a reduction result for block-transitive Steiner 2-designs whose automorphism groups are solvable.
In the early 2000s, the Li-Liu team carried out extensive research on block-transitive 2-designs; for details, see the Chinese survey by Tian et al.~\cite{Tian}.

The study of block-transitive designs for $t>2$ has also yielded a substantial body of research.
In 1993, Cameron and Praeger \cite{Cameron0} studied point-imprimitive block-transitive designs.
They also proved that a non-complete block-transitive $t$-design must satisfy $t\leq7$ \cite{Cameron1}.
In 2010, Huber \cite{Huber1} proved that no non-trivial Steiner $5$-design admits a block-transitive automorphism group
of affine type, and obtained the corresponding result for Steiner $4$-designs, except possibly when $G\leq\mathrm{A}\Gamma\mathrm{L}(1,q)$.
Between 2010 and 2020, Liu's research group sequentially discussed several special types of block-transitive designs, obtaining numerous significant results (see survey paper \cite{Tian}).
In 2023, Gan and Liu \cite{Gan} proved that a block-transitive, point-primitive automorphism group of a non-trivial Steiner $3$-design is of affine or almost simple type.
In the same year, Lan, Liu, and Yin \cite{Lan} classified block-transitive Steiner $3$-designs admitting an almost simple automorphism group with alternating socle.
In 2024, they further investigated block-transitive Steiner $3$-designs admitting an almost simple automorphism group
with socle a simple exceptional group of Lie type, reducing the possible socles to Suzuki groups and $\G_2(q)$ and obtaining further restrictions \cite{Lan1}.

In 1984, Camina and Gagen \cite{Camina1} investigated block-transitive Steiner $2$-designs satisfying $k$ divides $v$ and established the well-known Camina--Gagen Theorem.
The Camina--Gagen Theorem shows that the divisibility condition $k$ divides $v$ imposes strong restrictions on block-transitive designs.
Recently, by imposing the Camina-Gagen condition, we have studied several families of block-transitive designs \cite{Huang,Huang1}.
In 2026, Wang, Huang, Zhang, and Liu \cite{Wang2026} proved that there are no non-trivial block-transitive $6$-$(v,k,2)$ designs satisfying $k\mid v$.
In this paper, we investigate non-trivial block-transitive $5$-$(v,k,2)$ designs satisfying $k$ divides $v$.
Our main result is as follows.

\begin{theorem}\label{th1.1}
Let $\mathcal{D}=(\mathcal{P},\mathcal{B})$ be a non-trivial simple $5$-$(v,k,2)$ design satisfying $k$ divides $v$, and let $G\leq\Aut(\mathcal{D})$ act block-transitively on $\mathcal{B}$.
Then, up to isomorphism, $\mathcal{D}$ is one of the following designs:
\begin{enumerate}[(1)]
\item a $5$-$(12,6,2)$ design admitting $\P\G\L(2,11)$ as a block-transitive automorphism group;
\item a $5$-$(24,8,2)$ design admitting $\P\G\L(2,23)$ as a block-transitive automorphism group.
\end{enumerate}

Conversely, each of these two designs admits the indicated block-transitive automorphism group.
\end{theorem}

Note that the two designs appearing in Theorem \ref{th1.1} are not new constructions.
Wei and Li \cite{WeiLi} proved the uniqueness, up to isomorphism, of the block-transitive $5$-$(12,6,2)$ design arising from the natural action of $\P\G\L(2,11)$.
Betten, Laue and Wassermann \cite[p.\~85]{Betten} observed that the union of two Witt $5$-$(24,8,1)$ designs interchanged by $\P\G\L(2,23)$ is a $5$-$(24,8,2)$ design whose block set is a single $\P\G\L(2,23)$-orbit.
The main purpose of the present paper is to prove that no other non-trivial simple block-transitive $5$-$(v,k,2)$ designs satisfying $k$ divides $v$ can occur.

\section{Preliminaries}

In this section, we collect some preliminary results on designs and permutation groups that will be used in the proof of Theorem~\ref{th1.1}.
We first recall some well-known facts about designs.
\begin{lemma}{\rm\cite{Beth}}\label{lem2.1}
Let $\mathcal{D}=(\mathcal{P},\mathcal{B})$ be a non-trivial $t$-$(v,k,\lambda)$ design.
Then the following holds:
\begin{enumerate}[{\rm(i)}]

\item $bk=vr$;

\item ${\binom{v}{t} }\lambda=b{\binom{k}{t}}$;

\item $r(k-1)=\lambda_{2}(v-1)$ for $t\geq 2$.

\item $\displaystyle
\lambda_s
=\lambda\frac{\binom{v-s}{t-s}}{\binom{k-s}{t-s}}
$ for $0\leq s\leq t$.
\end{enumerate}
\end{lemma}

\begin{lemma}{\rm\cite{Beth}}\label{lem2.2}
Let $\mathcal{D}=(\mathcal{P},\mathcal{B})$ be a non-trivial $2$-$(v,k,\lambda)$ design with $b$ blocks.
Then $b\geq v$.
\end{lemma}

As a consequence of Lemmas~\ref{lem2.1} and~\ref{lem2.2}, we obtain the following upper bound for $k$.

\begin{corollary}\label{cor2.3}
Let $\mathcal{D}=(\mathcal{P},\mathcal{B})$ be a non-trivial $5$-$(v,k,2)$ design.
Then
\begin{equation*}
(k-3)(k-4)\leq 2(v-4),
\end{equation*}

and hence
\begin{equation*}
k\leq\left\lfloor\frac{7+\sqrt{8v-31}}{2}\right\rfloor<\sqrt{2v}+4.
\end{equation*}
\end{corollary}

\begin{proof}
Fix a $3$-subset $S$ of $\mathcal P$.
The derived design is a $2$-$(v-3,k-3,2)$ design.
By Fisher's inequality and Lemma~\ref{lem2.1},
\[
v-3\leq\lambda_3
=\frac{2(v-3)(v-4)}{(k-3)(k-4)}.
\]
Hence $(k-3)(k-4)\leq2(v-4)$.
Solving this inequality and using $k\in\mathbb Z$ gives
\[
k\leq\left\lfloor\frac{7+\sqrt{8v-31}}{2}\right\rfloor
<\sqrt{2v}+4.
\]
\end{proof}

\begin{theorem}{\rm\cite{Cameron1}}\label{th2.4}
Let $\mathcal{D}=(\mathcal{P},\mathcal{B})$ be a $t$-design with $t\geq2$.
Then the following statements hold:
\begin{enumerate}[{\rm(i)}]
\item If $G\leq\Aut(\mathcal{D})$ acts block-transitively on $\mathcal{D}$, then $G$ is point $\left\lfloor\frac{t}{2}\right\rfloor$-homogeneous on $\mathcal{P}$.
\item If $G\leq\Aut(\mathcal{D})$ acts flag-transitively on $\mathcal{D}$, then $G$ is point $\left\lfloor\frac{t+1}{2}\right\rfloor$-homogeneous on $\mathcal{P}$.
\end{enumerate}
\end{theorem}

Next, we give some lemmas which will be required in Section 3 for the proof of our main result.
Let $q=p^e$ be a prime power and $H$ a subgroup of $\P\S\L(2,q)$.
Furthermore, let $N_l$ denote the number of the orbits of length $l$ and let $n=(2,q-1)$.
For the action of subgroups of $\P\S\L(2,q)$ on the projective line, we refer to \cite{Cameron2}.
By Dickson's classification of the subgroups of $\P\S\L(2,q)$ {\rm\cite{Dickson}}, every non-trivial subgroup
of $\P\S\L(2,q)$ is of one of the types considered below.

\begin{lemma}{\rm\cite{Cameron2}}\label{lem2.5}
Let $H$ be the cyclic group of order $c$ with $c\mid\frac{q\pm1}{n}$.
Then we have
\begin{enumerate}[{\rm(i)}]

\item if $c\mid\frac{q+1}{n}$, then $N_c=(q+1)/c$;

\item if $c\mid\frac{q-1}{n}$, then $N_1=2$ and $N_c=(q-1)/c$.
\end{enumerate}
\end{lemma}

\begin{lemma}{\rm\cite{Cameron2}}\label{lem2.6}
Let $H$ be the dihedral group of order $2c$ with $c\mid\frac{q\pm1}{n}$.
Then
\begin{enumerate}[{\rm(i)}]

\item for $q\equiv1~(\mod4)$, we have
\begin{enumerate}[{\rm(a)}]
\item if $c\mid\frac{q+1}{2}$, then $N_c=2$ and $N_{2c}=(q+1-2c)/(2c)$,

\item if $c\mid\frac{q-1}{2}$, then $N_2=1$, $N_c=2$ and $N_{2c}=(q-1-2c)/(2c)$, unless $c=2$, in which case $N_2=3$ and $N_{4}=(q-5)/4$;
\end{enumerate}

\item for $q\equiv3~(\mod4)$, we have
\begin{enumerate}[{\rm(a)}]
\item if $c\mid\frac{q+1}{2}$, then $N_{2c}=(q+1)/(2c)$,

\item if $c\mid\frac{q-1}{2}$, then $N_2=1$ and $N_{2c}=(q-1)/(2c)$;
\end{enumerate}

\item for $q\equiv0~(\mod2)$, we have
\begin{enumerate}[{\rm(a)}]

\item if $c\mid q+1$, then $N_c=1$ and $N_{2c}=(q+1-c)/(2c)$,

\item if $c\mid q-1$, then $N_2=1$, $N_c=1$ and $N_{2c}=(q-1-c)/(2c)$.
\end{enumerate}
\end{enumerate}
\end{lemma}

\begin{lemma}{\rm\cite{Cameron2}}\label{lem2.7}
Let $H$ be the elementary abelian group of order $\bar{q}\mid q$.
Then we have $N_1=1$ and $N_{\bar{q}}=q/\bar{q}$.
\end{lemma}

\begin{lemma}{\rm\cite{Cameron2}}\label{lem2.8}
Let $H$ be a semi-direct product of the elementary abelian group of order $\bar{q}\mid q$ and the cyclic group of order $c>1$ with $c\mid \bar{q}-1$ and $c\mid q-1$.
Then we have $N_1=1$, $N_{\bar{q}}=1$ and $N_{c\bar{q}}=(q-\bar{q})/(c\bar{q})$.
\end{lemma}

\begin{lemma}{\rm\cite{Cameron2}}\label{lem2.9}
Let $H$ be $\P\S\L(2,\bar{q})$ with $\bar{q}^m=q$, $m\geq 1$.
Then  $N_{\bar{q}+1}=1$, $N_{\bar{q}(\bar{q}-1)}=1$ if $m$ is even, and all other orbits are regular.
\end{lemma}

\begin{lemma}{\rm\cite{Cameron2}}\label{lem2.10}
Let $H$ be $\P\G\L(2,\bar{q})$ with $\bar{q}^m=q$, $m>1$ even.
Then $N_{\bar{q}+1}=1$, $N_{\bar{q}(\bar{q}-1)}=1$, and all other orbits are regular.
\end{lemma}

\begin{lemma}{\rm\cite{Cameron2}}\label{lem2.11}
Let $H$ be isomorphic to $\A_4$.
Then
\begin{enumerate}[{\rm(i)}]

\item for $q\equiv1~(\mod4)$, we have
\begin{enumerate}[{\rm(a)}]
\item if $3\mid\frac{q+1}{2}$, then $N_6=1$ and $N_{12}=(q-5)/12$,

\item if $3\mid\frac{q-1}{2}$, then $N_4=2$, $N_6=1$ and $N_{12}=(q-13)/12$,

\item if $3\mid q$, then $N_4=1$, $N_6=1$ and $N_{12}=(q-9)/12$;
\end{enumerate}

\item for $q\equiv3~(\mod4)$, we have
\begin{enumerate}[{\rm(a)}]
\item if $3\mid\frac{q+1}{2}$, then $N_{12}=(q+1)/12$,

\item if $3\mid\frac{q-1}{2}$, then $N_4=2$ and $N_{12}=(q-7)/12$,

\item if $3\mid q$, then $N_4=1$ and $N_{12}=(q-3)/12$;
\end{enumerate}

\item for $q=2^e$, $e\equiv0~(\mod2)$, we have $N_1=1$, $N_4=1$ and $N_{12}=(q-4)/12$.
\end{enumerate}
\end{lemma}

\begin{lemma}{\rm\cite{Cameron2}}\label{lem2.12}
Let $H$ be isomorphic to $\S_4$.
Then
\begin{enumerate}[{\rm(i)}]
\item for $q\equiv1~(\mod8)$, we have
\begin{enumerate}[{\rm(a)}]
\item if $3\mid\frac{q+1}{2}$, then $N_6=1$, $N_{12}=1$ and $N_{24}=(q-17)/24$,

\item if $3\mid\frac{q-1}{2}$, then $N_6=1$, $N_8=1$, $N_{12}=1$ and $N_{24}=(q-25)/24$,

\item if $3\mid q$, then $N_4=1$, $N_6=1$ and $N_{24}=(q-9)/24$;
\end{enumerate}

\item for $q\equiv-1~(\mod8)$, we have
\begin{enumerate}[{\rm(a)}]
\item if $3\mid\frac{q+1}{2}$, then $N_{24}=(q+1)/24$,

\item if $3\mid\frac{q-1}{2}$, then $N_8=1$ and $N_{24}=(q-7)/24$.
\end{enumerate}
\end{enumerate}
\end{lemma}

\begin{lemma}{\rm\cite{Cameron2}}\label{lem2.13}
Let $H$ be isomorphic to $\A_5$.
Then
\begin{enumerate}[{\rm(i)}]
\item for $q\equiv1~(\mod4)$, we have
\begin{enumerate}[{\rm(a)}]
\item if $q=5^e$, $e\equiv1~(\mod2)$, then $N_6=1$ and $N_{60}=(q-5)/60$,

\item if $q=5^e$, $e\equiv0~(\mod2)$, then $N_6=1$, $N_{20}=1$ and $N_{60}=(q-25)/60$,

\item if $15\mid\frac{q+1}{2}$, then $N_{30}=1$ and $N_{60}=(q-29)/60$,

\item if $3\mid\frac{q+1}{2}$ and $5\mid\frac{q-1}{2}$, then $N_{12}=1$, $N_{30}=1$ and $N_{60}=(q-41)/60$,

\item if $3\mid\frac{q-1}{2}$ and $5\mid\frac{q+1}{2}$, then $N_{20}=1$, $N_{30}=1$ and $N_{60}=(q-49)/60$,

\item if $15\mid\frac{q-1}{2}$, then $N_{12}=1$, $N_{20}=1$, $N_{30}=1$ and $N_{60}=(q-61)/60$,

\item if $3\mid q$ and $5\mid\frac{q+1}{2}$, then $N_{10}=1$ and $N_{60}=(q-9)/60$,

\item if $3\mid q$ and $5\mid\frac{q-1}{2}$, then $N_{10}=1$, $N_{12}=1$ and $N_{60}=(q-21)/60$;
\end{enumerate}

\item for $q\equiv3~(\mod4)$, we have
\begin{enumerate}[{\rm(a)}]
\item if $15\mid\frac{q+1}{2}$, then $N_{60}=(q+1)/60$,

\item if $3\mid\frac{q+1}{2}$ and $5\mid\frac{q-1}{2}$, then $N_{12}=1$ and $N_{60}=(q-11)/60$,

\item if $3\mid \frac{q-1}{2}$ and $5\mid\frac{q+1}{2}$, then $N_{20}=1$ and $N_{60}=(q-19)/60$,

\item if $15\mid \frac{q-1}{2}$, then $N_{12}=1$, $N_{20}=1$ and $N_{60}=(q-31)/60$.
\end{enumerate}
\end{enumerate}
\end{lemma}

\begin{lemma}{\rm\cite{Camina1}}\label{lem2.14}
Let $G$ be a permutation group on a set $\Delta$.
For a prime $p$, assume that $p^e\parallel|\Delta|$ and $p^e\mid [G:G_{\alpha\beta}]$ for all $\alpha,\beta\in\Delta$, $\alpha\neq\beta$.
Then the length of each orbit of $G$ on $\Delta$ is divisible by $p^e$.
\end{lemma}

The following classification result on flag-transitive $5$-designs will also be used in the proof of Theorem~\ref{th1.1}.

\begin{lemma} {\rm\cite{GongFlag}}\label{lem2.15}
There exists no non-trivial flag-transitive $5$-$(v,k,2)$ design.
\end{lemma}

The following is a simple consequence of block-transitivity which will be used repeatedly in this paper, and is therefore stated here.

\begin{lemma}\label{lem2.16}
Let $\mathcal{D}=(\mathcal{P},\mathcal{B})$ be a design with $b$ blocks.
If $G\leq\Aut(\mathcal{D})$ acts block-transitively on $\mathcal{B}$, then, for every $B\in\mathcal{B}$, we have $b=[G:G_B]$.
In particular, $b\mid |G|$.
\end{lemma}

\begin{lemma}\label{lem2.17}
Let $\mathcal{D}=(\mathcal{P},\mathcal{B})$ be a $5$-$(v,k,2)$ design, and let $S$ be a $5$-subset of $\mathcal{P}$.
Let $B_1$ and $B_2$ be the two blocks containing $S$.
If $H\leq\Aut(\mathcal{D})$ stabilizes $S$ setwise, then $H$ permutes $B_1$ and $B_2$.
Consequently, $B_1\cup B_2$ is $H$-invariant, $|B_1\cup B_2|\leq2k$, and $[H:H_{B_i}]\leq2$ for $i=1,2$.
\end{lemma}

\begin{proof}
Since $H$ stabilizes $S$ setwise, for every $h\in H$ and $i\in\{1,2\}$, we have $S=h(S)\subseteq h(B_i)$.
As precisely two blocks contain $S$, it follows that $h(B_i)\in\{B_1,B_2\}$.
Thus $H$ permutes $B_1$ and $B_2$, and hence $B_1\cup B_2$ is $H$-invariant and $|B_1\cup B_2|\leq2k$.
Finally, the orbit of $B_i$ under $H$ has length at most $2$, so the orbit-stabilizer theorem gives $[H:H_{B_i}]=|B_i^H|\leq2$.
\end{proof}

\begin{lemma}\label{lem2.18}
Let $\mathcal{D}=(\mathcal{P},\mathcal{B})$ be a non-trivial $5$-$(v,k,2)$ design satisfying $k\mid v$, and let $G\leq\Aut(\mathcal{D})$ act block-transitively on $\mathcal{B}$.
Then $2\mid k$ or $3\mid k$.
\end{lemma}

\begin{proof}
Suppose, to the contrary, that $(k,6)=1$.
Let $B\in\mathcal{B}$ and let $\alpha,\beta\in B$ be distinct.
By Theorem~\ref{th2.4}, the block-transitivity of $G$ implies that $G$ is point-transitive.
Hence
\begin{equation*}
[G:G_{\alpha\beta B}]
=v[G_\alpha:G_{\alpha\beta B}]
=b[G_B:G_{\alpha\beta B}].
\end{equation*}

Together with the formula for $b$ in Lemma~\ref{lem2.1}, this gives
\begin{equation*}
2(v-1)(v-2)(v-3)(v-4)[G_B:G_{\alpha\beta B}]
=
k(k-1)(k-2)(k-3)(k-4)
[G_\alpha:G_{\alpha\beta B}].
\end{equation*}

Since $k\mid v$, we have $(k,v-i)=(k,i)$ for $1\leq i\leq4$.
As $(k,6)=1$, it follows that
\begin{equation*}
\bigl(k,2(v-1)(v-2)(v-3)(v-4)\bigr)=1.
\end{equation*}

Consequently,\(k\mid[G_B:G_{\alpha\beta B}]\)for all distinct $\alpha,\beta\in B$.

Let $p^e\parallel k$.
Then $p^e\parallel|B|$ and $p^e\mid[G_B:G_{\alpha\beta B}]$ for all distinct $\alpha,\beta\in B$.
By Lemma~\ref{lem2.14}, the length of every $G_B$-orbit on $B$ is divisible by $p^e$.
Applying this argument to every prime divisor of $k$, we obtain that every $G_B$-orbit on $B$ has length divisible by $k$.
Since $|B|=k$, the group $G_B$ is transitive on $B$.
Together with the block-transitivity of $G$, this implies that $G$ is flag-transitive, contrary to Lemma~\ref{lem2.15}.
Therefore, $2\mid k$ or $3\mid k$.
\end{proof}

\section{The Proof of the Main Theorem}
Let $\mathcal{D}=(\mathcal{P},\mathcal{B})$ be a non-trivial simple $5$-$(v,k,2)$ design satisfying $k\mid v$, and let
$G\leq\Aut(\mathcal{D})$ act block-transitively on $\mathcal{B}$.
By Theorem~\ref{th2.4}, the group $G$ is $2$-homogeneous on $\mathcal{P}$.

We now recall the classification of finite $2$-homogeneous permutation groups {\rm\cite{Kantor,Xu}}.
If $G$ is not $2$-transitive, then $v=q\equiv3\pmod4$ and $G\leq\A\Gamma\L(1,q)$, which is included in the affine case{\rm(A1)}.
If $G$ is $2$-transitive, then $G$ is either of affine type or of almost simple type.
We consider these two types separately.

\subsection{The affine case}
Suppose that $G$ is of affine type.
Then the point set $\mathcal{P}$ may be identified with the vector space $V=V(d,p)$, where $v=p^d$ for some prime $p$.
Since $k\mid v=p^d$ and $5<k<v$, we may write $k=p^i$, where $1\leq i<d$.
By Lemma~\ref{lem2.18}, we have $2\mid k$ or $3\mid k$, and hence $p\in\{2,3\}$.

Moreover,
\begin{equation*}
G=T\rtimes G_0\leq\A\G\L(d,p),
\end{equation*}
where $T$ is the regular elementary abelian translation group of $V$ and $G_0=G_{\mathbf{0}}$ is the stabilizer of the zero vector $\mathbf{0}$.
By the classification of finite affine $2$-homogeneous permutation groups {\rm\cite{Kantor,Xu}}, we may also identify $V(d,p)$ with $V(n,q)$, where $q=p^f$ and $d=nf$, and one of the following cases occurs:
\begin{enumerate}[{\rm(A}1{\rm)}]
\item $G\leq\A\Gamma\L(1,p^d)$;
\item $\S\L(n,q)\unlhd G_0$, where $n\geq2$;
\item $\Sp(2m,q)\unlhd G_0$, where $n=2m\geq4$;
\item $G_2(q)'\unlhd G_0$, where $n=6$ and $q$ is even;
\item $G_0\cong\A_6$ or $\A_7$, and $v=2^4$;
\item $G_0$ contains a normal subgroup isomorphic to $\S\L(2,3)$ or
$\S\L(2,5)$, and $v=3^4$;
\item $G_0$ has an extraspecial normal subgroup $E$ of order $2^5$,
$G_0/E$ is isomorphic to a subgroup of $\S_5$, and $v=3^4$;
\item $G_0\cong\S\L(2,13)$ and $v=3^6$.
\end{enumerate}

Throughout this subsection, we repeatedly use Lemma~\ref{lem2.1}, Corollary~\ref{cor2.3}, and the divisibility condition in Lemma~\ref{lem2.16}.
In particular,
\begin{equation*}
\lambda_4=\frac{2(v-4)}{k-4},\qquad
\lambda_3=\frac{2(v-3)(v-4)}{(k-3)(k-4)},\qquad
\lambda_2=\frac{2(v-2)(v-3)(v-4)}
{(k-2)(k-3)(k-4)}
\end{equation*}
are positive integers, and
\begin{equation*}
(k-3)(k-4)\leq2(v-4),\qquad
k<\sqrt{2v}+4,\qquad
b=\frac{2\binom{v}{5}}{\binom{k}{5}}\mid|G|.
\end{equation*}

We use Lemma~\ref{lem2.17} in the orbit arguments below and consider Cases~{\rm(A1)}--{\rm(A8)} separately.

\begin{lemma}\label{lem3.1}
The case {\rm(A1)} is impossible.
\end{lemma}

\begin{proof}
Let $q=p^d=v$.
Since $b=2\binom{q}{5}/\binom{k}{5}\leq|G|\leq q(q-1)d$, Corollary~\ref{cor2.3} gives $(q-2)(q-3)\leq d\,k(k-1)(k-2)$.
Suppose that $q\geq4096$.
Then $(q-2)(q-3)\geq(255/256)^2q^2$ and $k<(17/16)\sqrt{2q}$.
Together with $d\leq\log_2q$, these inequalities yield
\begin{equation*}
\frac{\sqrt q}{\log_2q}
<
\frac{(17/16)^3\,2^{3/2}}{(255/256)^2}
<3.42.
\end{equation*}

However, the function $f(x)=\sqrt{x}/\log_2x$ is increasing for $x\geq4096$, and its value at $x=4096$ is $64/12>5.33$, a contradiction.

Hence $q<4096$.
Recall that $p\in\{2,3\}$ and $k=p^i$, where $1\leq i<d$.
Checking the integrality of $\lambda_4$, $\lambda_3$, $\lambda_2$ and $b$, we obtain the following four possibilities:
\begin{center}
\begin{tabular}{c|c|c|c}
\hline
$q$ & $k$ & $b$ & $q(q-1)d$\\
\hline
$64$ & $8$ & $272304$ & $24192$\\
$128$ & $8$ & $9448800$ & $113792$\\
$1024$ & $8$ & $331828071168$ & $10475520$\\
$2048$ & $8$ & $10670587975168$ & $46114816$\\
\hline
\end{tabular}
\end{center}

In each case, $b>q(q-1)d\geq|G|$, contrary to $b\mid|G|$.
Therefore, this case cannot occur.
\end{proof}

\begin{lemma}\label{lem3.2}
The case {\rm(A2)} is impossible.
\end{lemma}

\begin{proof}
Let $v=q^n$, where $q$ is a power of $2$ or $3$.
Since $G_0$ normalizes $\S\L(n,q)$ in its natural representation, we have $G\leq\A\Gamma\L(n,q)$.
Thus $G$ preserves affine dimension and parallelism over $\mathbb F_q$.
First consider the following cases:
\begin{center}
\begin{tabular}{c|c|c}
\hline
Parameter range & $\dim E$ & Orbit length on $V\setminus E$\\
\hline
$q\geq8,\ n\geq3$ & $1$ & $q^n-q$\\
$q=3,4,\ n\geq4$ & $2$ & $q^n-q^2$\\
$q=2,\ n\geq5$ & $3$ & $q^n-8$\\
\hline
\end{tabular}
\end{center}

In each case, choose a subspace $E$ of the indicated dimension and a $5$-subset $S\subseteq E$.
The pointwise stabilizer of $E$ in $\S\L(n,q)$ is transitive on $V\setminus E$.
By Corollary~\ref{cor2.3}, its orbit length exceeds $2k$.
Lemma~\ref{lem2.17} therefore implies that both blocks containing $S$ lie in $E$.
By block-transitivity, every block lies in an affine subspace of dimension $\dim E$.
But some $5$-subset of $V$ has affine span of larger dimension and thus lies in no block, a contradiction.

Next suppose that $n=2$ and $q\geq9$.
Let $L$ be a line through $0$, choose a $5$-subset $S\subseteq L$, and let $U$ be the transvection group fixing $L$ pointwise.
Then $U\cong(\mathbb F_q,+)$, and its orbits on $V\setminus L$ are the affine lines parallel to $L$, each of length $q$.

If $q$ is odd, then $|U|$ is odd, so $U$ fixes each of the two blocks containing $S$.
Since $k<\sqrt{2}q+4<2q$, each block lies in the union of at most two parallel lines.
By block-transitivity, this holds for every block.
However, each affine line meets the parabola $\{(t,t^2):t\in\mathbb F_q\}$ in at most two points.
Thus five points on this parabola cannot lie in any block, a contradiction.

If $q$ is even, then $q=2^e\geq16$.
Let $K$ be the kernel of the action of $U$ on the two blocks containing $S$.
Then $[U:K]\leq2$, and every $K$-orbit on $V\setminus L$ has length at least $q/2$.
If either block meets three such orbits, then
\begin{equation*}
k\geq5+\frac{3q}{2}>\sqrt{2}q+4,
\end{equation*}
contrary to Corollary~\ref{cor2.3}.
Hence each block meets at most two $K$-orbits outside $L$.
Each orbit lies in a line parallel to $L$, so both blocks lie in unions of at most three parallel lines.
By block-transitivity, this holds for every block.

Now choose
\begin{equation*}
S=\{(t,t^2):t\in\{0,\alpha_1,\alpha_2,\alpha_3,\alpha_4\}\},
\end{equation*}
where $\alpha_1,\ldots,\alpha_4\in\mathbb F_q$ are linearly independent over $\mathbb F_2$.
The slope between $(t,t^2)$ and $(u,u^2)$ is $t+u$.
These ten slopes are distinct, so three parallel lines contain at most four points of $S$.
Thus $S$ lies in no block, a contradiction.

It remains to consider $(q,n)\in\{(2,3),(2,4),(3,2),(3,3),(4,2),(4,3),(8,2)\}$.
The condition $k\mid v$ and the integrality of the derived parameters leave only $(v,k)=(64,8)$, arising from$(q,n)=(4,3)$ or $(8,2)$.
Then $b=272304=2^4\cdot3^2\cdot31\cdot61$, whereas $|G|$ divides $64\cdot2|\G\L(3,4)|$ or $64\cdot3|\G\L(2,8)|$, respectively.
Neither number is divisible by $31$ or $61$, so $b\nmid|G|$, a contradiction.
This excludes Case~{\rm(A2)}.
\end{proof}

\begin{lemma}\label{lem3.3}
The case {\rm(A3)} is impossible.
\end{lemma}

\begin{proof}
Let $(V,\beta)$ be the natural $2m$-dimensional symplectic space over $\mathbb F_q$, and put $L=\Sp(V)\unlhd G_0$.
Thus $v=q^{2m}$.
By Corollary~\ref{cor2.3}, we have $k<\sqrt{2}\,q^m+4$.

We first consider the following parameter ranges.
In each case, choose a subspace $E$ as indicated, and let $\ell(E)$ denote the minimum length of an orbit of the pointwise stabilizer $L_{(E)}$ on $V\setminus E$:
\begin{center}
\begin{tabular}{c|c|c}
\hline
Parameter range & $E$ & $\ell(E)$\\
\hline
$q\geq8,\ m\geq2$ & $\langle x\rangle,\quad x\neq0$ & $q^{2m-1}-q$\\
$q=3,\ m\geq4$ & $\text{a non-degenerate plane}$ & $q^{2m-2}-1$\\
$q=4,\ m\geq3$ & $\text{a non-degenerate plane}$ & $q^{2m-2}-1$\\
$q=2,\ m\geq5$ & $\text{a \(3\)-space with }\dim(E\cap E^\perp)=1$ & $2^{2m-3}-2$\\
\hline
\end{tabular}
\end{center}

Next, we explain how the orbit lengths in the last column are obtained.
Suppose first that $E=\langle x\rangle$.
By Witt's lemma \cite{KleidmanLiebeck}, $L_{(E)}=L_x$ is transitive on $x^\perp\setminus\langle x\rangle$, which has size $q^{2m-1}-q$, and is transitive on each set $\{y\in V\mid\beta(x,y)=c\}$, where $c\in\mathbb F_q^\times$; each such set has size $q^{2m-1}$.
This proves the first orbit bound.

If $E$ is a non-degenerate plane, then $V=E\perp E^\perp$ and $L_{(E)}\cong\Sp(E^\perp)$.
Every vector in $V\setminus E$ has the form $e+w$, where $e\in E$ and $0\neq w\in E^\perp$.
Since $\Sp(E^\perp)$ is transitive on the nonzero vectors of $E^\perp$, every such orbit has length $q^{2m-2}-1$.

Finally, suppose that $q=2$.
Choose a non-degenerate plane $U$ and a nonzero vector $x\in U^\perp$, and set $E=U\perp\langle x\rangle$.
Then $E\cap E^\perp=\langle x\rangle$.
Writing $W=U^\perp$, we have $L_{(E)}\cong\Sp(W)_x$, where $\dim W=2m-2$.
Applying the one-dimensional calculation above to $W$ shows that every $L_{(E)}$-orbit on $V\setminus E$ has length at least $2^{2m-3}-2$.

In each of the four parameter ranges, the corresponding lower bound satisfies $\ell(E)>3q^m+8$.
Indeed, the weakest cases are $(q,m)=(8,2),(3,4),(4,3)$, and $(2,5)$, respectively.
Since $\sqrt2<3/2$, Corollary~\ref{cor2.3} gives $2k<3q^m+8<\ell(E)$.

Choose a $5$-subset $S$ of $E$, and let $B_1$ and $B_2$ be the two blocks containing $S$.
The group $L_{(E)}$ fixes $S$ pointwise, so Lemma~\ref{lem2.17} shows that $B_1\cup B_2$ is $L_{(E)}$-invariant and has size at most $2k$.
Since every $L_{(E)}$-orbit on $V\setminus E$ has length greater than $2k$, such an orbit cannot meet $B_1\cup B_2$.
Hence $B_1,B_2\subseteq E$.

Under the identification used in Case~{\rm(A3)}, the elements of $G$ map affine $\mathbb F_q$-subspaces to affine $\mathbb F_q$-subspaces of the same dimension.
As a consequence of block-transitivity, it follows that every block is contained in an affine subspace of dimension at most $3$.
However, since $2m\geq4$, the space $V$ contains five affinely independent points.
Their affine span has dimension $4$, so they cannot be contained in any block.
This contradicts the definition of $5$-design.

It remains to consider $(q,m)\in\{(2,2),(2,3),(2,4),(3,2),(3,3),(4,2)\}$.
For these six pairs, checking the divisors $k$ of $v=q^{2m}$ with $5<k<v$ against the integrality conditions
$\lambda_s=2\binom{v-s}{5-s}/\binom{k-s}{5-s}\in\mathbb Z$, $0\leq s\leq4$, leaves only $(q,m,v,k)=(2,3,64,8)$.

For this remaining possibility,$b=272304=2^4\cdot3^2\cdot31\cdot61$.
But $G\leq\A\G\L(6,2)$ and
\begin{equation*}
|\A\G\L(6,2)|
=2^{21}\prod_{j=1}^{6}(2^j-1)
\end{equation*}
is not divisible by $61$.
Thus $b\nmid|G|$, contrary to Lemma~\ref{lem2.16}.
Thus Case~{\rm(A3)} is impossible.
\end{proof}

\begin{lemma}\label{lem3.4}
The case {\rm(A4)} is impossible.
\end{lemma}

\begin{proof}
Let $q=2^e$ and $v=q^6$.
Suppose first that $q\geq8$.
Choose a nonzero vector $x\in V$, put $E=\langle x\rangle$, and choose a $5$-subset $S$ of $E$.
Let $B_1$ and $B_2$ be the two blocks containing $S$.
Put $L=G_2(q)$, $P=L_E$, and $H=L_x$.
Since $q\geq8$, we have $L=G_2(q)'\leq G_0$.
Moreover, $H\unlhd P$ and $[P:H]\mid q-1$, because $H$ is
the kernel of the action of $P$ on the one-dimensional space $E$.

In the natural action on the projective points of $V$, the group
$P$ has four orbits, namely $\{E\}$ and three further orbits
$\Omega_1,\Omega_2,\Omega_3$ of lengths
\begin{equation*}
q(q+1),\qquad q^3(q+1),\qquad q^5,
\end{equation*}
respectively; see \cite{DempwolffKantor}.
In the standard embedding of the split Cayley hexagon,
the $q+1$ hexagon lines through $E$ lie in a projective plane
and form all lines of that plane through $E$ \cite{ThasVanMaldeghem}.
Their points other than $E$ form $\Omega_1$.
Hence $\{E\}\cup\Omega_1$ is the set of all one-dimensional
subspaces of a three-dimensional $\mathbb F_q$-subspace $U$.

Since $q$ is even, each of the three numbers
$q(q+1)$, $q^3(q+1)$, and $q^5$ is coprime to $q-1$.
As $H\unlhd P$, the number of $H$-orbits on each $\Omega_i$
divides both $[P:H]$ and $|\Omega_i|$.
Consequently, $H$ is transitive on each $\Omega_i$.

For any $y\in V\setminus U$, the projective point $\langle y\rangle$
belongs to $\Omega_2$ or $\Omega_3$.
The map $z\mapsto\langle z\rangle$ sends the $H$-orbit of $y$
onto the $H$-orbit of $\langle y\rangle$.
Hence
\begin{equation*}
|y^H|
\geq |\langle y\rangle^H|
\geq \min\{q^3(q+1),q^5\}
=q^3(q+1).
\end{equation*}
By Corollary~\ref{cor2.3}, for $q\geq8$ we have
\begin{equation*}
q^3(q+1)>2\sqrt{2}\,q^3+8>2k.
\end{equation*}

Since $H$ fixes $E$ pointwise, it fixes $S$ pointwise.
By Lemma~\ref{lem2.17}, the set $B_1\cup B_2$ is $H$-invariant and has size at most $2k$.
Thus it cannot contain any point of $V\setminus U$, and therefore$B_1,B_2\subseteq U$.

Under the natural embedding $G\leq\mathrm{A\Gamma L}(6,q)$ in Case~{\rm(A4)}, elements of $G$ preserve affine $\mathbb F_q$-dimension.
Due to block-transitivity, every block is necessarily contained in an affine $3$-dimensional subspace.
This is impossible, since five affinely independent points cannot be contained in such a subspace.

If $q=2$, then $v=64$, and the parameter conditions leave only $k=8$.
In this case, $b=2^4\cdot3^2\cdot31\cdot61$, but $G\leq\A\G\L(6,2)$ and $61\nmid|\A\G\L(6,2)|$.
Thus $b\nmid|G|$, a contradiction.
If $q=4$, then $v=4096$.
A direct check using $k\mid v$ and the integrality of $\lambda_4$, $\lambda_3$ and $\lambda_2$ shows that no possible value of $k$ exists.
Therefore, this case cannot occur.
\end{proof}

\begin{lemma}\label{lem3.5}
Cases {\rm(A5)}--{\rm(A8)} are impossible.
\end{lemma}

\begin{proof}
In Case~{\rm(A5)}, we have $G_0\cong\A_6$ or $\A_7$ and $v=16$.
Since $5<k<16$ and $k\mid16$, it follows that $k=8$.
However, $\lambda_3=2\cdot13\cdot12/(5\cdot4)=78/5$ is not an integer, a contradiction.

In Cases~{\rm(A6)} and~{\rm(A7)}, we have $v=81$ and hence $k\in\{9,27\}$.
The corresponding values of $\lambda_4$ are $154/5$ and $154/23$, neither of which is an integer.
Thus both cases are excluded.

Finally, in Case~{\rm(A8)}, we have $G_0\cong\S\L(2,13)$ and $v=3^6=729$.
Thus $k\in\{9,27,81,243\}$.
For $k\in\{27,81,243\}$, the value $\lambda_4=1450/(k-4)$ is not an integer.
For $k=9$, we have $\lambda_2=727\cdot726\cdot725/105$, which is not an integer since $7$ divides the denominator but none of the factors in the numerator.
Hence none of Cases~{\rm(A5)}--{\rm(A8)} can occur.
\end{proof}

By Lemmas~\ref{lem3.1}--\ref{lem3.5}, all eight affine cases are excluded.
Consequently, $G$ cannot be of affine type.

\subsection{The almost simple case}

We now assume that $G$ is of almost simple type.
Let $N$ be the socle of $G$.
Then $N$ is a non-abelian simple group satisfying $N\leq G\leq\Aut(N)$.
By the classification of finite almost simple $2$-homogeneous permutation groups{\rm\cite{Kantor,Xu}}, one of the following possibilities for $(N,v)$ occurs:
\begin{enumerate}[{\rm(B}1{\rm)}]
\item $N=\A_v$, where $v\geq5$;
\item $N=\P\S\L(d,q)$ and $v=(q^d-1)/(q-1)$, where
$d\geq2$ and $(d,q)\neq(2,2),(2,3)$;
\item $N=\P\S\U(3,q)$ and $v=q^3+1$, where $q>2$;
\item $N=\Sz(q)$ and $v=q^2+1$, where $q=2^{2e+1}>2$;
\item $N=\Re(q)$ and $v=q^3+1$, where $q=3^{2e+1}>3$;
\item $N=\Sp(2d,2)$ and $v=2^{2d-1}\pm2^{d-1}$, where $d\geq3$;
\item $N=\P\S\L(2,11)$ with $v=11$, or
$N=\P\S\L(2,8)$ with $v=28$;
\item $N=\M_v$, where $v\in\{11,12,22,23,24\}$;
\item $N=\M_{11}$ with $v=12$, $N=\A_7$ with $v=15$,
$N=\HS$ with $v=176$, or $N=\Co_3$ with $v=276$.
\end{enumerate}

We now consider Cases~{\rm(B1)}--{\rm(B9)} in turn.

\begin{lemma}
Case (B1) is impossible.
\end{lemma}

\begin{proof}
Since $\mathcal D$ is non-trivial, we have $5<k<v$, and hence$v\geq7$. 
The group $N=A_v$ is transitive on the set of all $k$-subsets of $\mathcal P$. Let $B\in\mathcal B$. 
Since $N\leq G\leq\Aut(\mathcal D)$, every member of the orbit $B^N$ is a block of $\mathcal D$. It follows that every $k$-subset of $\mathcal P$ is a block, and hence $\mathcal D$ is the complete$5$-design.
 Therefore,
\begin{equation*}
2=\binom{v-5}{k-5}.
\end{equation*}
Since $1\leq k-5<v-5$, this implies that $(v,k)=(7,6)$, contrary to $k\mid v$. 
Therefore, this case cannot occur.
\end{proof}

We next consider Case (B2). 
Thus $N=\P\S\L(d,q)$ and $v=(q^d-1)/(q-1)$, where $d\geq2$ and $(d,q)\neq(2,2),(2,3)$. We treat the possible configurations separately according to $d$, the characteristic of the underlying field, and the position of $G$ between $N$ and $\Aut(N)$.

We first record a common reduction for the case $d=2$.

\begin{remark}
Suppose that $N=\P\S\L(2,q)$, where $q=p^e>3$, and let $B\in\mathcal B$.
Put $n=(2,q-1)$. Since $N\unlhd G$ and $G$ is block-transitive on $\mathcal B$, all $N$-orbits on $\mathcal B$ have the same length $[N:N_B]=[\P\S\L(2,q):\P\S\L(2,q)_B]$.
Hence
\begin{equation*}
\frac{b}{[N:N_B]}
=
\frac{2n(q-2)(q-3)|\P\S\L(2,q)_B|}
{k(k-1)(k-2)(k-3)(k-4)}
\end{equation*}
is a positive integer. 
Since$(k-1)(k-2)(k-3)(k-4)/(2n)$ is an integer, it follows that$k\mid(q-2)(q-3)|\P\S\L(2,q)_B|$.
Moreover, $k\mid q+1$, so$q\equiv-1\pmod{k}$ and hence
\begin{equation*}
k\mid12|\P\S\L(2,q)_B|.
\end{equation*}

Suppose now that $|\P\S\L(2,q)_B|>1$. 
Then$\P\S\L(2,q)_B$ is a non-trivial subgroup of $\P\S\L(2,q)$,and, since it stabilizes $B$ setwise, the block $B$ is a unionof $\P\S\L(2,q)_B$-orbits on $\mathbb P^1(q)$. 
Let $z$ be a non-negative integer. By Lemmas 2.5--2.13, together with $k\mid q+1$ and $5<k<q+1$, the following reduced possibilities remain.

\begin{enumerate}

\item If $\P\S\L(2,q)_B\cong C_c$, where $c\geq2$, then $k=cz$ when $c\mid(q+1)/n$, while $k=cz$, $cz+1$ or $cz+2$ when $c\mid(q-1)/n$.

\item If $\P\S\L(2,q)_B\cong D_{2c}$, then, for $c=2$,$k=4z+2$ when $q\equiv1\pmod4$, while $k=4z$ when$q\equiv3\pmod4$. 
For $c\geq3$, if $q\equiv1\pmod4$ and $c\mid(q+1)/2$, then $k=2cz$ or $2cz+c$; 
if $q\equiv1\pmod4$ and $c\mid(q-1)/2$, then $k=2cz+2$ or$2cz+c+2$;
if $q\equiv3\pmod4$ and $c\mid(q+1)/2$, then$k=2cz$;
if $q\equiv3\pmod4$ and $c\mid(q-1)/2$, then$k=2cz+2$. 
If $q$ is even and $c\mid q+1$, then$k=2cz+c$, while if $c\mid q-1$, then $k=2cz+c+2$.

\item If $\P\S\L(2,q)_B\cong E_{\bar q}$, where $\bar q\mid q$,
then $k=\bar qz+1$.

\item If $\P\S\L(2,q)_B\cong E_{\bar q}\rtimes C_c$, where
$\bar q\mid q$, $c\mid(\bar q-1)$ and $c\mid(q-1)$, then
$k=c\bar qz+1$ or $c\bar qz+\bar q+1$.

\item If $\P\S\L(2,q)_B\cong\P\S\L(2,\bar q)$, where
$q=\bar q^m$ and $m>1$, then
$k=(\bar q^3-\bar q)z/n+\bar q+1$ when $m$ is odd, while
$k=(\bar q^3-\bar q)z/n+\bar q^2+1$ when $m$ is even.

\item If $\P\S\L(2,q)_B\cong\P\G\L(2,\bar q)$, where
$q=\bar q^m$, $m>1$ and $m$ is even, then
$k=(\bar q^3-\bar q)z+\bar q^2+1$.

\item If $\P\S\L(2,q)_B\cong A_4$, then $k\mid144$.
If $q$ is even, Lemma~\ref{lem2.11} gives
$k=12z+1$ or $12z+5$. Thus $\gcd(k,6)=1$, contrary to
Lemma~\ref{lem2.18}.

Suppose that $q$ is odd. By Lemma~\ref{lem2.11} and
$k\mid q+1$, the only possibilities are
\begin{center}
\begin{tabular}{c|c}
\hline
$q\pmod{12}$ & $k$\\
\hline
$3$ & $16$\\
$5$ & $6,18$\\
$7$ & $8,16$\\
$11$ & $12,24,36,48,72,144$\\
\hline
\end{tabular}
\end{center}
There is no possible value of $k$ when
$q\equiv1$ or $9\pmod{12}$.

\item If $\P\S\L(2,q)_B\cong S_4$, then $q$ is odd and
$k\mid288$. By Lemma 2.12 and $k\mid q+1$, the only
possibilities are
\begin{center}
\begin{tabular}{c|c}
\hline
$q\pmod{24}$ & $k$\\
\hline
$7$ & $8,32$\\
$17$ & $6,18$\\
$23$ & $24,48,72,96,144,288$\\
\hline
\end{tabular}
\end{center}
There is no possible value of $k$ in the remaining congruence
classes occurring in Lemma 2.12.

\item If $\P\S\L(2,q)_B\cong A_5$ and $q$ is odd, then
$k\mid720$. By Lemma 2.13 and $k\mid q+1$, the only
possibilities are

\begin{center}
\begin{tabular}{c|c}
\hline
\text{condition on }q & k\\
\hline
$q=5^e,\ e>1\text{ odd}$ & $6$\\
$q\equiv9\pmod{60}$ & $10$\\
$q\equiv11\pmod{60}$ & $12,72$\\
$q\equiv19\pmod{60}$ & $20,80$\\
$q\equiv29\pmod{60}$ & $30,90$\\
$q\equiv59\pmod{60}$ & $60,120,180,240,360,720$\\
\hline
\end{tabular}
\end{center}

There is no possible value of $k$ in the remaining cases of
Lemma 2.13. If $q$ is even, then
$A_5\cong\P\S\L(2,4)$ is already included in the subfield
subgroup case in {\rm (5)}.

\end{enumerate}
\end{remark}

We first consider the case $G=N$.

\begin{lemma}
Case (B2) with $d=2$ and $G=N=\P\S\L(2,q)$, where $q=p^e>3$,
is impossible.
\end{lemma}

\begin{proof}
In this case, $\mathcal P=\mathbb P^1(q)=GF(q)\cup\{\infty\}$ and $|G|=(q+1)q(q-1)/n$, where $n=(2,q-1)$. Since $\mathcal D$ is non-trivial, we have $5<k<q+1$. 
For any $B\in\mathcal B$, block-transitivity gives $b=[G:G_B]=[\P\S\L(2,q):\P\S\L(2,q)_B]$. 
Combining this with Lemma 2.1 yields
\begin{equation}\tag{3.1}\label{eq3.1}
2n(q-2)(q-3)|\P\S\L(2,q)_B|
=
k(k-1)(k-2)(k-3)(k-4).
\end{equation}

Suppose first that $|\P\S\L(2,q)_B|=1$.
Since$(k-1)(k-2)(k-3)(k-4)/(2n)$ is an integer, equation (3.1)implies $k\mid(q-2)(q-3)$.
Together with $k\mid q+1$ and$((q-2)(q-3),q+1)\mid12$, this gives $k\mid12$, and hence$k=6$ or $12$. For $k=6$, equation (3.1) gives$q^2-5q+6=360/n$, while for $k=12$ it gives$q^2-5q+6=47520/n$.
For $n=1,2$, the discriminant of each quadratic equation is not a square.
Hence$|\P\S\L(2,q)_B|=1$ is impossible.

We may therefore assume that$|\P\S\L(2,q)_B|>1$.
By the preceding remark, the possible structures of $\P\S\L(2,q)_B$ and the corresponding forms of$k$ are precisely those listed in {\rm (1)}--{\rm (9)}.
We consider them in turn.

\vspace{2mm}

\textbf{Case 1}: $\P\S\L(2,q)_B\cong C_c$, where $c\geq2$.

\vspace{2mm}

Assume first that $k=cz$.
Substituting $|\P\S\L(2,q)_B|=c$ into equation~\eqref{eq3.1} and cancelling the common factor $c$, we obtain
\begin{equation}\tag{3.2}\label{eq3.2}
2n(q-2)(q-3)
=z(k-1)(k-2)(k-3)(k-4).
\end{equation}
Since $(k-1)(k-2)(k-3)(k-4)/(2n)$ is an integer, equation~\eqref{eq3.2} implies that $z\mid(q-2)(q-3)$.
Moreover, $z\mid k$ and $k\mid q+1$, and hence $z\mid q+1$.
Since $(q-2)(q-3)\equiv12\pmod{q+1}$, it follows that $z\mid12$.

We first consider the case where $c\mid(q+1)/n$.
If $q$ is odd, then $n=2$ and $z\in\{1,2,3,4,6,12\}$.
Since $c\mid(q+1)/2$, equation~\eqref{eq3.2} gives
\begin{equation*}
c\mid\frac{(q+1)(q-6)}{2}
=\frac{cz^2}{8}\bigl(c^3z^3-10c^2z^2+35cz-50\bigr)+3z-6.
\end{equation*}
After clearing the possible denominator in the first term, we obtain $c\mid3z-6$ when $z$ is even, while $c\mid12z-24$ when $z$ is odd.

The case $z=2$ requires separate consideration, since the preceding divisibility condition imposes no restriction on $c$.
In this case, $k=2c$, and Lemma~\ref{lem2.1} gives $\lambda_4=(q-3)/(c-2)\in\mathbb{Z}$.
Combining this equality with equation~\eqref{eq3.2}, we obtain
\begin{equation*}
\lambda_4\bigl(\lambda_4(c-2)+1\bigr)
=(2c-1)(2c-2)(2c-3).
\end{equation*}
Reducing this equality modulo $c-2$ gives $\lambda_4\equiv6\pmod{c-2}$.
Moreover, the displayed equality implies that $\lambda_4>6$.
Thus $(\lambda_4-6)/(c-2)$ is a positive integer.
Substituting $\lambda_4=6+(c-2)(\lambda_4-6)/(c-2)$ into the displayed equality gives
\begin{equation*}
(c-2)^2\left(\left(\frac{\lambda_4-6}{c-2}\right)^2-8\right)
+12(c-2)\left(\frac{\lambda_4-6}{c-2}-2\right)
+\frac{\lambda_4-6}{c-2}+14=0.
\end{equation*}
The left-hand side is negative when $(\lambda_4-6)/(c-2)=1$ and positive when $(\lambda_4-6)/(c-2)\geq3$.
Hence $(\lambda_4-6)/(c-2)=2$.
It follows that $(c-2)^2=4$, and hence $c=4$.
Consequently, $\lambda_4=10$, $q=23$ and $k=8$.

The preceding calculation, together with direct substitution for the remaining values of $z$, gives the following table.

\begin{center}
\begin{tabular}{c|c|c}
\hline
$z$ & Possible values of $c$ & Remaining $(q,k,c)$\\
\hline
$1$  & $6,12$                  & $(47,12,12)$\\
$2$  & $c\geq3$                & $(23,8,4)$\\
$3$  & $2,3,4,6,12$            & ---\\
$4$  & $2,3,6$                 & ---\\
$6$  & $2,3,4,6,12$            & ---\\
$12$ & $2,3,5,6,10,15,30$      & ---\\
\hline
\end{tabular}
\end{center}

If $q$ is even, then $n=1$.
Since $k\mid q+1$, the integers $k$, $c$ and $z$ are all odd.
The left-hand side of equation~\eqref{eq3.2} is divisible by $2^2$ but not by $2^3$, whereas $(k-1)(k-2)(k-3)(k-4)$ is divisible by $2^3$.
This is a contradiction.
Therefore, no possibility occurs when $q$ is even.

We next consider the case where $c\mid(q-1)/n$ and $k=cz$.
Since $c\mid k$ and $k\mid q+1$, we also have $c\mid q+1$.
If $q$ is odd, then $n=2$ and $c\mid((q-1)/2,q+1)$, which forces $c=2$.
If $q$ is even, then $c\mid(q-1,q+1)=1$, contrary to $c\geq2$.
Thus $q$ is odd and $c=2$.

In this case, $k=2z$.
Since $z\mid12$ and $k>5$, we have $z\in\{3,4,6,12\}$ and $k\in\{6,8,12,24\}$.
Substitution into equation~\eqref{eq3.2} gives the following results.

\begin{center}
\begin{tabular}{c|c|c}
\hline
$z$ & $k$ & Result\\
\hline
$3$  & $6$  & $q=12$\\
$4$  & $8$  & ---\\
$6$  & $12$ & ---\\
$12$ & $24$ & ---\\
\hline
\end{tabular}
\end{center}

Here and below, a dash means that the corresponding equation has no integral solution.
The value $q=12$ contradicts the assumption that $q$ is odd.
Therefore, no possibility occurs in this case.

It remains to consider $k=cz+1$ or $k=cz+2$.
By Lemma~\ref{lem2.5}, these two possibilities occur only when $c\mid(q-1)/n$.
If $k=cz+1$, then equation~\eqref{eq3.1} becomes
\begin{equation*}
2n(q-2)(q-3)
=kz(k-2)(k-3)(k-4).
\end{equation*}
If $k=cz+2$, then equation~\eqref{eq3.1} becomes
\begin{equation*}
2n(q-2)(q-3)
=kz(k-1)(k-3)(k-4).
\end{equation*}
Since $(k-3)(k-4)/2$ is an integer, both equations imply that $k\mid n(q-2)(q-3)$.
Together with $k\mid q+1$, this yields $k\mid12n$.

If $q$ is even, then $n=1$ and $k\mid12$.
However, $k>5$ and $k\mid q+1$, where $q+1$ is odd, so no possibility occurs.
Hence $q$ is odd, and therefore $n=2$ and $k\in\{6,8,12,24\}$.
The possible values of $(c,z)$ and the corresponding solutions of equation~\eqref{eq3.1} are listed below.

\begin{center}
\begin{tabular}{c|c|c|c}
\hline
Form of $k$ & $k$ & Possible $(c,z)$ & Result\\
\hline
$cz+1$ & $6$  & $(5,1)$                  & ---\\
$cz+1$ & $8$  & $(7,1)$                  & $q=18$\\
$cz+1$ & $12$ & $(11,1)$                 & ---\\
$cz+1$ & $24$ & $(23,1)$                 & ---\\
\hline
$cz+2$ & $6$  & $(2,2),(4,1)$            & $q=12$ if $(c,z)=(2,2)$\\
$cz+2$ & $8$  & $(2,3),(3,2),(6,1)$      & ---\\
$cz+2$ & $12$ & $(2,5),(5,2),(10,1)$     & ---\\
$cz+2$ & $24$ & $(2,11),(11,2),(22,1)$   & ---\\
\hline
\end{tabular}
\end{center}

The only integral values obtained are $q=12$ and $18$, both of which contradict the assumption that $q$ is odd.
Consequently, this case leaves only the two arithmetic candidates $(q,k,c)=(23,8,4)$ and $(q,k,c)=(47,12,12)$.

\vspace{2mm}

\textbf{Case 2}: $\P\S\L(2,q)_B\cong D_{2c}$, where $c\geq2$.

\vspace{2mm}

We first consider the case $c=2$.
Since $c\mid(q\pm1)/n$, the case where $q$ is even cannot occur.
Suppose that $q\equiv1\pmod4$.
Then $c\mid(q-1)/2$, and Lemma~\ref{lem2.6}, together with $k\mid q+1$, gives $k=4z+2$.
Substituting this expression into equation~\eqref{eq3.1}, we obtain
\begin{equation*}
2(q-2)(q-3)
=zk(k-1)\frac{(k-3)(k-4)}{n}.
\end{equation*}
Since $(k-3)(k-4)/n$ is an integer, it follows that $k\mid2(q-2)(q-3)$.
Together with $k\mid q+1$, this gives $k\mid24$.
Since $k>5$ and $k\equiv2\pmod4$, we have $k=6$.
Equation~\eqref{eq3.1} then gives $q^2-5q-39=0$, which has no integral solution.

Suppose now that $q\equiv3\pmod4$.
Then $c\mid(q+1)/2$, and Lemma~\ref{lem2.6} gives $k=4z$.
Equation~\eqref{eq3.1} becomes
\begin{equation*}
(q-2)(q-3)
=z\frac{(k-1)(k-2)(k-3)(k-4)}{2n}.
\end{equation*}
Since the quotient on the right-hand side is an integer, we have $z\mid(q-2)(q-3)$.
Moreover, $z\mid k\mid q+1$, and hence $z\mid12$.
Since $k>5$, it remains to consider $z\in\{2,3,4,6,12\}$.
The corresponding calculations are listed below.

\begin{center}
\begin{tabular}{c|c|c}
\hline
$z$ & $k$ & Result\\
\hline
$2$  & $8$  & $q=23$\\
$3$  & $12$ & ---\\
$4$  & $16$ & ---\\
$6$  & $24$ & ---\\
$12$ & $48$ & ---\\
\hline
\end{tabular}
\end{center}

Here and below, a dash means that the corresponding equation has no integral solution.
Thus the case $c=2$ leaves only the arithmetic candidate $(q,k,c)=(23,8,2)$.

We may now assume that $c\geq3$.
We first consider the possibility $k=2cz$.
Substituting $|\P\S\L(2,q)_B|=2c$ into equation~\eqref{eq3.1} and cancelling the common factor $2c$, we obtain
\begin{equation}\tag{3.3}\label{eq3.3}
(q-2)(q-3)
=z\frac{(k-1)(k-2)(k-3)(k-4)}{2n}.
\end{equation}
Since $(k-1)(k-2)(k-3)(k-4)/(2n)$ is an integer, equation~\eqref{eq3.3} implies that $z\mid(q-2)(q-3)$.
Moreover, $z\mid k$ and $k\mid q+1$, so $z\mid q+1$.
Since $(q-2)(q-3)\equiv12\pmod{q+1}$, it follows that $z\mid12$.

In this case, $q$ must be odd, since $k=2cz$ is even whereas $k\mid q+1$.
Thus $n=2$, $c\mid(q+1)/2$, and $z\in\{1,2,3,4,6,12\}$.
By equation~\eqref{eq3.3},
\begin{equation*}
c\mid\frac{(q+1)(q-6)}{2}
=\frac{cz^2}{2}\bigl(4c^3z^3-20c^2z^2+35cz-25\bigr)+3z-6.
\end{equation*}
Hence $c\mid3z-6$ when $z$ is even, while $c\mid6z-12$ when $z$ is odd.

The case $z=2$ requires separate consideration.
In this case, $k=4c$, and $(q+1)/(4c)$ is an integer greater than one.
By Lemma~\ref{lem2.1},
\begin{equation*}
\lambda_4=\frac{q-3}{2(c-1)}
\end{equation*}
is an integer, and hence
\begin{equation*}
c-1\mid2\left(\frac{q+1}{4c}-1\right).
\end{equation*}
On the other hand, equation~\eqref{eq3.3} gives
\begin{equation*}
2\sqrt2\,\frac{(c-1)^2}{c}
<
\frac{q+1}{4c}
<
2\sqrt2\,c+\frac1c.
\end{equation*}
If $c\geq13$, these two inequalities imply
\begin{equation*}
5<
\frac{2}{c-1}\left(\frac{q+1}{4c}-1\right)
<6,
\end{equation*}
contrary to the preceding divisibility condition.
Therefore, $3\leq c\leq12$.
The possible values of $(q+1)/(4c)$ satisfying both the inequalities and the divisibility condition are listed below.

\begin{center}
\begin{tabular}{c|c|c}
\hline
$c$ & Possible values of $(q+1)/(4c)$ & Result\\
\hline
$3$  & $4,5,6,7,8$ & ---\\
$4$  & $7,10$       & ---\\
$5$  & $11,13$      & ---\\
$6$  & $16$         & ---\\
$7$  & $16,19$      & ---\\
$8$  & $22$         & ---\\
$9$  & $21,25$      & ---\\
$10$ & $28$         & ---\\
$11$ & $26,31$      & ---\\
$12$ & $34$         & ---\\
\hline
\end{tabular}
\end{center}

Thus the case $z=2$ cannot occur.
For the remaining values of $z$, the preceding divisibility conditions and equation~\eqref{eq3.3} give the following table.

\begin{center}
\begin{tabular}{c|c|c}
\hline
$z$ & Possible values of $c$ & Remaining $(q,k,c)$\\
\hline
$1$  & $3,6$                  & $(47,12,6)$\\
$3$  & $3,6$                  & ---\\
$4$  & $3,6$                  & ---\\
$6$  & $3,4,6,12$             & ---\\
$12$ & $3,5,6,10,15,30$       & ---\\
\hline
\end{tabular}
\end{center}

We next consider the possibility $k=2cz+c=c(2z+1)$.
Substituting this expression into equation~\eqref{eq3.1} and cancelling $c$, we obtain
\begin{equation}\tag{3.4}\label{eq3.4}
4n(q-2)(q-3)
=(2z+1)(k-1)(k-2)(k-3)(k-4).
\end{equation}
Since $(k-1)(k-2)(k-3)(k-4)/(4n)$ is an integer, we have $2z+1\mid(q-2)(q-3)$.
Moreover, $2z+1\mid k\mid q+1$, and hence $2z+1\mid12$.
Since $2z+1$ is odd, we have $z=0$ or $1$.

Suppose first that $z=0$.
Then $k=c$.
If $q$ is odd, then $q\equiv1\pmod4$, $n=2$, and $c\mid(q+1)/2$.
In particular, $c$ is odd.
Reducing equation~\eqref{eq3.4} modulo $c$ gives $c\mid72$.
Since $c=k>5$, it follows that $c=9$.
Equation~\eqref{eq3.1} then gives $q=17$, but
\begin{equation*}
\lambda_4=\frac{2(q-3)}{k-4}=\frac{28}{5},
\end{equation*}
a contradiction.
If $q$ is even, then $c$ is odd and the same argument gives $c\mid24$, which is impossible since $c>5$.
Therefore, $z=0$ cannot occur.

Suppose now that $z=1$.
Then $k=3c$.
Reducing equation~\eqref{eq3.4} modulo $c$ gives $c\mid24$.
Since $c$ is odd and $c\geq3$, we have $c=3$ and $k=9$.
If $q$ is odd, equation~\eqref{eq3.1} has no integral solution.
If $q$ is even, equation~\eqref{eq3.1} gives $q=38$, which is not a prime power.
Thus the possibility $k=2cz+c$ is excluded.

We now consider the possibility $k=2cz+2$.
By Lemma~\ref{lem2.6}, this can occur only when $q$ is odd and $c\mid(q-1)/2$.
Equation~\eqref{eq3.1} gives
\begin{equation*}
2(q-2)(q-3)
=zk(k-1)\frac{(k-3)(k-4)}{n}.
\end{equation*}
Since $(k-3)(k-4)/n$ is an integer, it follows that $k\mid2(q-2)(q-3)$.
Together with $k\mid q+1$, this gives $k\mid24$.
Since $c\geq3$ and $k=2cz+2>5$, the remaining possibilities are listed below.

\begin{center}
\begin{tabular}{c|c|c}
\hline
$k$ & Possible $(c,z)$ & Result\\
\hline
$8$  & $(3,1)$  & ---\\
$12$ & $(5,1)$  & ---\\
$24$ & $(11,1)$ & ---\\
\hline
\end{tabular}
\end{center}

Therefore, the possibility $k=2cz+2$ cannot occur.

Finally, consider the possibility $k=2cz+c+2=c(2z+1)+2$.
Equation~\eqref{eq3.1} gives
\begin{equation*}
4(q-2)(q-3)
=(2z+1)k(k-1)\frac{(k-3)(k-4)}{n}.
\end{equation*}
Since $(k-3)(k-4)/n$ is an integer, it follows that $k\mid4(q-2)(q-3)$.
Together with $k\mid q+1$, this gives $k\mid48$.
Since $c\geq3$, the remaining possibilities are as follows.

\begin{center}
\begin{tabular}{c|c|c}
\hline
$k$ & Possible $(c,z)$ & Result\\
\hline
$6$  & $(4,0)$  & ---\\
$8$  & $(6,0)$  & ---\\
$12$ & $(10,0)$ & ---\\
$16$ & $(14,0)$ & ---\\
$24$ & $(22,0)$ & ---\\
$48$ & $(46,0)$ & ---\\
\hline
\end{tabular}
\end{center}

Therefore, the possibility $k=2cz+c+2$ cannot occur.
Consequently, this case leaves only the two arithmetic candidates $(q,k,c)=(23,8,2)$ and $(q,k,c)=(47,12,6)$.

\vspace{2mm}

\textbf{Case 3}: $\P\S\L(2,q)_B\cong E_{\bar q}$, where $\bar q\mid q$.

\vspace{2mm}

In this case, Lemma~\ref{lem2.7} gives $k=\bar qz+1$.
Substituting $|\P\S\L(2,q)_B|=\bar q$ and $k-1=\bar qz$ into equation~\eqref{eq3.1}, and cancelling the common factor $\bar q$, we obtain
\begin{equation*}
2(q-2)(q-3)
=zk\frac{(k-2)(k-3)(k-4)}{n}.
\end{equation*}
Since $(k-2)(k-3)(k-4)$ is the product of three consecutive integers and $n\in\{1,2\}$, the quotient $(k-2)(k-3)(k-4)/n$ is an integer.
It follows that $k\mid2(q-2)(q-3)$.
On the other hand, $k\mid q+1$, and hence $q\equiv-1\pmod{k}$.
Therefore,
\begin{equation*}
k\mid2(q-2)(q-3)\equiv24\pmod{k},
\end{equation*}
so that $k\mid24$.
Since $k>5$, we have $k\in\{6,8,12,24\}$.

Moreover, $\bar q\mid k-1$ and $\bar q>1$.
Since $k-1$ is prime for each of the four possible values of $k$, it follows that $\bar q=k-1$ and $z=1$.
In particular, $\bar q$ and hence $q$ are odd, so $n=2$.
Substitution into the preceding equation gives the following table.

\begin{center}
\begin{tabular}{c|c|c|c|c}
\hline
$k$ & $\bar q$ & $z$ & $(q-2)(q-3)$ & Result\\
\hline
$6$  & $5$  & $1$ & $36$    & The discriminant is $145$\\
$8$  & $7$  & $1$ & $240$   & $q=18$, not a power of $7$\\
$12$ & $11$ & $1$ & $2160$  & The discriminant is $8641$\\
$24$ & $23$ & $1$ & $55440$ & The discriminant is $221761$\\
\hline
\end{tabular}
\end{center}

None of the discriminants $145$, $8641$ and $221761$ is a perfect square, while the remaining case gives $q=18$, which is not a power of $7$.
Consequently, this case cannot occur.

\vspace{2mm}

\textbf{Case 4}: $\P\S\L(2,q)_B\cong E_{\bar q}\rtimes C_c$, where $\bar q\mid q$, $c\geq2$, $c\mid\bar q-1$ and $c\mid q-1$.

\vspace{2mm}

By Lemma~\ref{lem2.8}, we have $k=c\bar qz+1$ or $k=c\bar qz+\bar q+1$.

Assume first that $k=c\bar qz+1$.
Substituting $|\P\S\L(2,q)_B|=c\bar q$ and $k-1=c\bar qz$ into equation~\eqref{eq3.1}, and cancelling the common factor $c\bar q$, we obtain
\begin{equation*}
2(q-2)(q-3)
=zk\frac{(k-2)(k-3)(k-4)}{n}.
\end{equation*}
Since $(k-2)(k-3)(k-4)/n$ is an integer, it follows that $k\mid2(q-2)(q-3)$.
Together with $k\mid q+1$, this gives $k\mid24$.
Since $k>5$, we have $k\in\{6,8,12,24\}$.
For each of these values, $k-1\in\{5,7,11,23\}$ is prime.
However, $k-1=c\bar qz$, where $c\geq2$, $\bar q\geq2$ and $z\geq1$.
This is impossible.

We next consider the possibility $k=c\bar qz+\bar q+1$.
In this case,
\begin{equation*}
k-1=\bar q(cz+1)
\quad\text{and}\quad
k-2=c\bar qz+\bar q-1.
\end{equation*}
Since $c\mid\bar q-1$, we have $c\mid k-2$, and hence $c\bar q\mid(k-1)(k-2)$.
Cancelling the common factor $c\bar q$ from equation~\eqref{eq3.1}, we obtain
\begin{equation*}
2(q-2)(q-3)
=k(cz+1)\frac{k-2}{c}\frac{(k-3)(k-4)}{n}.
\end{equation*}
Both $(k-2)/c$ and $(k-3)(k-4)/n$ are integers.
Therefore, $k\mid2(q-2)(q-3)$.
Together with $k\mid q+1$, this again gives $k\mid24$, and hence
$k\in\{6,8,12,24\}$.

Now $k-1=\bar q(cz+1)$, and each possible value of $k-1$ is prime.
Since $\bar q>1$, it follows that $\bar q=k-1$ and $cz+1=1$.
Thus $z=0$ and $k=\bar q+1$.
In particular, $\bar q\in\{5,7,11,23\}$.
Since $\bar q\mid q$, the integer $q$ is odd, and hence $n=2$.

Substituting $k=\bar q+1$ into equation~\eqref{eq3.1} and cancelling $\bar q$, we obtain
\begin{equation*}
4c(q-2)(q-3)
=(\bar q+1)(\bar q-1)(\bar q-2)(\bar q-3).
\end{equation*}
Since $c\geq2$ and $c\mid\bar q-1$, the possible values of $c$ and the corresponding discriminants of the resulting quadratic equation in $q$ are as follows.

\begin{center}
\begin{tabular}{c|c|c}
\hline
$\bar q$ & $k$ & Possible $(c,\text{discriminant})$\\
\hline
$5$  & $6$  & $(2,73),(4,37)$\\
$7$  & $8$  & $(2,481),(3,321),(6,161)$\\
$11$ & $12$ & $(2,4321),(5,1729),(10,865)$\\
$23$ & $24$ & $(2,110881),(11,20161),(22,10081)$\\
\hline
\end{tabular}
\end{center}

None of the discriminants in the table is a perfect square.
Hence the corresponding quadratic equation has no integral solution for $q$.
Consequently, this case cannot occur.

\vspace{2mm}

\textbf{Case 5}:
$\P\S\L(2,q)_B\cong\P\S\L(2,\bar q)$, where $q=\bar q^m$ and
$m\geq1$, or
$\P\S\L(2,q)_B\cong\P\G\L(2,\bar q)$, where $q=\bar q^m$ and
$m>1$ is even.

\vspace{2mm}

We first consider
$\P\S\L(2,q)_B\cong\P\S\L(2,\bar q)$.
If $m=1$, then $\bar q=q$ and hence
$\P\S\L(2,q)_B=\P\S\L(2,q)$.
Since $\P\S\L(2,q)$ is transitive on $\mathbb{P}^{1}(q)$, the block
$B$ must be either empty or equal to the whole point set, contrary
to $5<k<q+1$.
Therefore, we may assume that $m>1$.

Since $q=\bar q^m$, we have
$n=(2,q-1)=(2,\bar q-1)$ and
$|\P\S\L(2,\bar q)|=(\bar q^3-\bar q)/n$.
Equation~\eqref{eq3.1} therefore gives
\begin{equation}\tag{3.5}\label{eq3.5}
2(q-2)(q-3)(\bar q^3-\bar q)
=k(k-1)(k-2)(k-3)(k-4).
\end{equation}
Since $k\mid q+1$, we have $q\equiv-1\pmod{k}$ and hence
$(q-2)(q-3)\equiv12\pmod{k}$.
It follows from equation~\eqref{eq3.5} that
$k\mid24(\bar q^3-\bar q)$.

Suppose first that $m$ is odd.
By Lemma~\ref{lem2.9}, together with $k\mid q+1$, we have
\begin{equation*}
k=\frac{\bar q^3-\bar q}{n}z+\bar q+1
=(\bar q+1)\left(\frac{\bar q(\bar q-1)}{n}z+1\right).
\end{equation*}
Since
$\bigl(\bar q(\bar q-1)z/n+1,\bar q(\bar q-1)/n\bigr)=1$,
the divisibility $k\mid24(\bar q^3-\bar q)$ implies that
$\bar q(\bar q-1)z/n+1\mid24n$.

If $z=0$, then $k=\bar q+1$.
Substituting this into equation~\eqref{eq3.5} and cancelling the
factor $\bar q^3-\bar q$, we obtain
$2(q-2)(q-3)=(\bar q-2)(\bar q-3)$.
For $\bar q=2$ or $3$, the right-hand side is zero whereas the
left-hand side is positive.
For $\bar q\geq4$, the inequality $q=\bar q^m>\bar q$ gives
$2(q-2)(q-3)>(\bar q-2)(\bar q-3)$.
Thus $z=0$ is impossible.

We may therefore assume that $z\geq1$.
The divisibility $\bar q(\bar q-1)z/n+1\mid24n$ leaves only
$(\bar q,z)=(2,1),(3,1),(3,5)$.
Substitution into equation~\eqref{eq3.5} gives the following table.
\begin{center}
\begin{tabular}{c|c|c|c}
\hline
$\bar q$ & $z$ & $k$ & Result\\
\hline
$2$ & $1$ & $9$  & $q=38$, not a power of $2$\\
$3$ & $1$ & $16$ & $q=107$, not a power of $3$\\
$3$ & $5$ & $64$ & Discriminant $76245121$, not a square\\
\hline
\end{tabular}
\end{center}
Therefore, the case where $m$ is odd cannot occur.

Suppose now that $m$ is even.
By Lemma~\ref{lem2.9}, together with $k\mid q+1$, we have
$k=(\bar q^3-\bar q)z/n+\bar q^2+1$.
Moreover,
$\bigl((\bar q^3-\bar q)/n,\bar q^2+1\bigr)=n$, and hence
$\bigl(k,(\bar q^3-\bar q)/n\bigr)=n$.
It follows from $k\mid24(\bar q^3-\bar q)$ that
$k/n\mid24n$.

Assume first that $z\geq1$.
If $\bar q\geq7$, then
$k/n\geq\bigl((\bar q^3-\bar q)/n+\bar q^2+1\bigr)/n>24n$,
contrary to $k/n\mid24n$.
Thus it remains to consider $\bar q\in\{2,3,4,5\}$.
The corresponding divisibility conditions are listed below.
\begin{center}
\begin{tabular}{c|c|c}
\hline
$\bar q$ & $k$ & Divisibility condition\\
\hline
$2$ & $6z+5$   & $6z+5\mid24$\\
$3$ & $12z+10$ & $6z+5\mid48$\\
$4$ & $60z+17$ & $60z+17\mid24$\\
$5$ & $60z+26$ & $30z+13\mid48$\\
\hline
\end{tabular}
\end{center}
None of these conditions admits a positive integer $z$.
Therefore, $z\geq1$ is impossible.

It remains to consider $z=0$.
Then $k=\bar q^2+1$.
Substituting this into equation~\eqref{eq3.5} and cancelling the
common factor $\bar q^3-\bar q$, we obtain
\begin{equation}\tag{3.6}\label{eq3.6}
2(q-2)(q-3)
=\bar q(\bar q^2+1)(\bar q^2-2)(\bar q^2-3).
\end{equation}
Since $k\mid q+1$, we have
$\bar q^2+1\mid\bar q^m+1$.
If $m\equiv0\pmod4$, then
$\bar q^m+1\equiv2\pmod{\bar q^2+1}$, a contradiction.
Consequently, $m\equiv2\pmod4$, and hence
$q=\bar q^m\equiv-1\pmod{\bar q^2+1}$.

Reducing equation~\eqref{eq3.6} modulo $\bar q^2+1$, its left-hand
side is congruent to $24$, whereas its right-hand side is congruent
to zero.
Thus $\bar q^2+1\mid24$.
For $\bar q=2,3,4$, the values of $\bar q^2+1$ are $5,10,17$,
none of which divides $24$, while $\bar q^2+1>24$ for
$\bar q\geq5$.
Therefore,
$\P\S\L(2,q)_B\cong\P\S\L(2,\bar q)$ cannot occur.

We next consider
$\P\S\L(2,q)_B\cong\P\G\L(2,\bar q)$, where $q=\bar q^m$ and
$m>1$ is even.
By Lemma~\ref{lem2.10}, the group $\P\G\L(2,\bar q)$ has one orbit
of length $\bar q+1$, one orbit of length $\bar q(\bar q-1)$, and
all its remaining orbits are regular.
Hence the possible values of $k$ are
$(\bar q^3-\bar q)z$,
$(\bar q^3-\bar q)z+\bar q+1$,
$(\bar q^3-\bar q)z+\bar q(\bar q-1)$, and
$(\bar q^3-\bar q)z+\bar q^2+1$.

The first and third possibilities are divisible by $\bar q$.
Since $\bar q\mid q$ and $(q,q+1)=1$, they are incompatible with
$k\mid q+1$.
For the second possibility, we have
$k=(\bar q+1)\bigl(\bar q(\bar q-1)z+1\bigr)$.
However, since $m$ is even,
$\bar q^m+1\equiv2\pmod{\bar q+1}$.
Thus $\bar q+1\nmid q+1$, and this possibility is also excluded.
Consequently, it remains only to consider
$k=(\bar q^3-\bar q)z+\bar q^2+1$.

Since $|\P\G\L(2,\bar q)|=\bar q^3-\bar q$,
equation~\eqref{eq3.1} gives
\begin{equation}\tag{3.7}\label{eq3.7}
2n(q-2)(q-3)(\bar q^3-\bar q)
=k(k-1)(k-2)(k-3)(k-4).
\end{equation}
Since $k\mid q+1$, we have $q\equiv-1\pmod{k}$ and hence
$(q-2)(q-3)\equiv12\pmod{k}$.
It follows from equation~\eqref{eq3.7} that
$k\mid24n(\bar q^3-\bar q)$.

Since $m$ is even, we have
$n=(2,q-1)=(2,\bar q-1)$.
Moreover,
$(\bar q^3-\bar q,\bar q^2+1)=n$, and hence
$(k,\bar q^3-\bar q)=n$.
Combining this equality with
$k\mid24n(\bar q^3-\bar q)$ gives
$k/n\mid24n$.

Assume first that $z\geq1$.
Then $k\geq\bar q^3-\bar q+\bar q^2+1$.
If $\bar q\geq7$, then $k/n>24n$, contrary to $k/n\mid24n$.
It remains to consider $\bar q\in\{2,3,4,5\}$.
The corresponding divisibility conditions are listed below.
\begin{center}
\begin{tabular}{c|c|c}
\hline
$\bar q$ & $k$ & Divisibility condition\\
\hline
$2$ & $6z+5$    & $6z+5\mid24$\\
$3$ & $24z+10$  & $12z+5\mid48$\\
$4$ & $60z+17$  & $60z+17\mid24$\\
$5$ & $120z+26$ & $60z+13\mid48$\\
\hline
\end{tabular}
\end{center}
None of these divisibility conditions admits a positive integer $z$.
Therefore, $z\geq1$ is impossible.

It remains to consider $z=0$.
Then $k=\bar q^2+1$.
Substituting this into equation~\eqref{eq3.7} and cancelling the
common factor $\bar q^3-\bar q$, we obtain
\begin{equation}\tag{3.8}\label{eq3.8}
2n(q-2)(q-3)
=\bar q(\bar q^2+1)(\bar q^2-2)(\bar q^2-3).
\end{equation}
Since $k\mid q+1$, we have
$\bar q^2+1\mid\bar q^m+1$.
If $m\equiv0\pmod4$, then
$\bar q^m+1\equiv2\pmod{\bar q^2+1}$, a contradiction.
Thus $m\equiv2\pmod4$, and hence
$q=\bar q^m\equiv-1\pmod{\bar q^2+1}$.

Reducing equation~\eqref{eq3.8} modulo $\bar q^2+1$, we obtain
$\bar q^2+1\mid24n$.
If $\bar q$ is even, then $n=1$, so
$\bar q^2+1\mid24$.
Since $\bar q^2+1$ is an odd integer at least $5$, this is
impossible.
If $\bar q$ is odd, then $n=2$, so
$\bar q^2+1\mid48$.
For $\bar q=3$, we have $\bar q^2+1=10\nmid48$, while for
$\bar q\geq5$ we have $\bar q^2+1\geq26$, and no such value
divides $48$.
Therefore,
$\P\S\L(2,q)_B\cong\P\G\L(2,\bar q)$ cannot occur.

Consequently, neither of the two subfield subgroup cases can occur.

\vspace{2mm}

\textbf{Case 6}: $\P\S\L(2,q)_B\cong\A_4$, $\S_4$ or $\A_5$.

\vspace{2mm}

The case where $\P\S\L(2,q)_B\cong\A_4$ and $q$ is even has
already been excluded in the preceding remark. If $q$ is even, then
an $\S_4$ subgroup does not occur, while an $\A_5$ subgroup is a
subfield subgroup isomorphic to $\P\S\L(2,4)$ and has already been
considered in Case~5. Hence $q$ is odd and $n=2$ in all the remaining
exceptional cases. Equation~\eqref{eq3.1} can therefore be written as
\begin{equation*}
4|\P\S\L(2,q)_B|(q-2)(q-3)
=
k(k-1)(k-2)(k-3)(k-4).
\end{equation*}

By the preceding remark, before applying equation~\eqref{eq3.1}, the
possible values of $k$ are
\begin{center}
\begin{tabular}{c|c}
\hline
$\P\S\L(2,q)_B$ & $k$\\
\hline
$\A_4$ & $6,8,12,16,18,24,36,48,72,144$\\
$\S_4$ & $6,8,18,24,32,48,72,96,144,288$\\
$\A_5$ & $6,10,12,20,30,60,72,80,90,120,180,240,360,720$\\
\hline
\end{tabular}
\end{center}
These lists can be reduced further by considering the $2$-parts of
equation~\eqref{eq3.1}. If $\P\S\L(2,q)_B\cong\A_4$ and $16\mid k$,
then $k\mid q+1$ gives $q\equiv-1\pmod{16}$, and hence
$q-3\equiv12\pmod{16}$. Thus the $2$-part of
$48(q-2)(q-3)$ is exactly $2^6$, whereas
$k(k-1)(k-2)(k-3)(k-4)$ is divisible by $2^7$, a contradiction.
Hence $16\nmid k$. Similarly, if
$\P\S\L(2,q)_B\cong\S_4$ and $32\mid k$, then the $2$-part of
$96(q-2)(q-3)$ is exactly $2^7$, whereas the right-hand side is
divisible by $2^8$, so $32\nmid k$. Finally, if
$\P\S\L(2,q)_B\cong\A_5$ and $16\mid k$, the same argument applied
to $240(q-2)(q-3)$ shows that the left-hand side has $2$-part
$2^6$, while the right-hand side is divisible by $2^7$. Hence again
$16\nmid k$.

Consequently, it is sufficient to consider
\begin{center}
\begin{tabular}{c|c}
\hline
$\P\S\L(2,q)_B$ & $k$\\
\hline
$\A_4$ & $6,8,12,18,24,36,72$\\
$\S_4$ & $6,8,18,24,48,72,144$\\
$\A_5$ & $6,10,12,20,30,60,72,90,120,180,360$\\
\hline
\end{tabular}
\end{center}
Substituting these values into equation~\eqref{eq3.1}, the only
integral values of $q$ are listed below.

\begin{center}
\begin{tabular}{c|c|c}
\hline
$\P\S\L(2,q)_B$ & $k$ & $q$\\
\hline
$\A_4$ & $12$ & $47$\\
$\S_4$ & $24$ & $233$\\
$\A_5$ & $60$ & $1655$\\
\hline
\end{tabular}
\end{center}

For $\P\S\L(2,q)_B\cong\A_4$, we have
$47\equiv11\pmod{12}$ and $12\mid48$, so the parameters are
compatible with the corresponding orbit structure in
Lemma~\ref{lem2.11}. Thus this case leaves the arithmetic candidate
$(q,k)=(47,12)$.

For $\P\S\L(2,q)_B\cong\S_4$, we have
$233\equiv17\pmod{24}$. By the preceding remark, when
$q\equiv17\pmod{24}$ the only possible values are $k=6$ and $18$;
in particular, $k=24$ cannot occur. Equivalently,
$24\nmid q+1$. Hence this case is impossible.

For $\P\S\L(2,q)_B\cong\A_5$, the value
$q=1655=5\cdot331$ is not a prime power. Moreover,
$1655\not\equiv59\pmod{60}$, whereas the preceding remark shows
that $k=60$ can occur only when $q\equiv59\pmod{60}$. Hence this
case is impossible.

Consequently, Case~6 leaves only the arithmetic candidate
$(q,k)=(47,12)$ arising from
$\P\S\L(2,q)_B\cong\A_4$.

Thus the remaining parameter sets are $5$-$(24,8,2)$ and
$5$-$(48,12,2)$. For the former, the possible block stabilizers are
$C_4$ and $D_4$ in $\P\S\L(2,23)$. For the latter, the possible
block stabilizers are $C_{12}$, $D_{12}$ and $\A_4$ in
$\P\S\L(2,47)$. Here $D_r$ denotes the dihedral group of order $r$.
We now exclude these remaining arithmetic candidates.

Recall that every $4$-subset of a $5$-$(v,k,2)$ design is contained
in exactly $\lambda_4=2(v-4)/(k-4)$ blocks. Thus $\lambda_4=10$ for
a $5$-$(24,8,2)$ design and $\lambda_4=11$ for a
$5$-$(48,12,2)$ design.

Put $N=\P\S\L(2,q)$. For each prescribed subgroup type, we take
representatives $H$ of all $N$-conjugacy classes and enumerate all
unions of $H$-orbits on $\mathbb P^1(q)$ having size $k$.
We retain a subset $B$ precisely when its full setwise stabilizer
$N_B$ equals $H$, and then remove duplicate $N$-orbits.
This procedure is exhaustive: any block with stabilizer conjugate
to $H$ has an $N$-translate that is a union of $H$-orbits and has
full stabilizer $H$.
There are respectively $1,2,1,2,2$ subgroup conjugacy classes and
$1,2,1,4,4$ candidate block orbits for the five rows below.
The enumeration and containment counts were computed using Magma.
The Magma code for this computation and for the verification of
Lemma~\ref{lem3.12}, together with the corresponding computational
output, is available from the authors upon request.

For each representative $B$, put $\mathcal O_B=B^N$ and define
\[
n_B(S)=|\{C\in\mathcal O_B:S\subseteq C\}|
      =\frac{|\{g\in N:S\subseteq B^g\}|}{|N_B|}
\]
for each $4$-subset $S$.
The last column of the following table gives two distinct values
of $n_B(S)$. In the $D_{12}$ and $\A_4$ rows, each displayed pair
occurs in two of the four candidate $N$-orbits. In every other
row, the displayed pair occurs in each candidate $N$-orbit.

\begin{center}
\begin{tabular}{c|c|c|c|c|c}
\hline
$(q,k)$ & $N_B$ & Form of $B$ & $N$-orbits
& $\lambda_4$ & Containment numbers\\
\hline
$(23,8)$ & $C_4$ & union of two $4$-orbits & $1$ & $10$ & $7,14$\\
$(23,8)$ & $D_4$ & union of two $4$-orbits & $2$ & $10$ & $4,12$\\
$(47,12)$ & $C_{12}$ & one $12$-orbit & $1$ & $11$ & $4,15$\\
$(47,12)$ & $D_{12}$ & one $12$-orbit & $4$ & $11$
& $7,14$ or $8,19$\\
$(47,12)$ & $\A_4$ & one $12$-orbit & $4$ & $11$
& $8,13$ or $8,14$\\
\hline
\end{tabular}
\end{center}

In every candidate block orbit, the containment number of a
$4$-subset is not constant and, in particular, is not always equal
to the required value of $\lambda_4$.
Hence none of these block orbits forms a $4$-design, and therefore
none can form a $5$-design. Consequently, when
$G=N=\P\S\L(2,q)$, there exists no non-trivial block-transitive
$5$-$(v,k,2)$ design satisfying $k\mid v$.
\end{proof}

\begin{lemma}\label{lem3.8}
Suppose that Case (B2) holds with $d=2$ and
$\P\S\L(2,q)<G\leq\P\Gamma\L(2,q)$, where $q=p^e>3$ and
$p>3$. Then
\begin{equation*}
(v,k,G)=(12,6,\P\G\L(2,11))
\quad\text{or}\quad
(v,k,G)=(24,8,\P\G\L(2,23)).
\end{equation*}
Moreover, both possibilities occur.
\end{lemma}

\begin{proof}
If $\P\S\L(2,q)$ is block-transitive on $\mathcal B$, then
Lemma~3.7 applies to $\P\S\L(2,q)$, which is impossible.
Hence we assume that $\P\S\L(2,q)$ is not block-transitive on
$\mathcal B$. Since $p>3$, the integer $q$ is odd and hence $n=2$.

Let
\begin{equation*}
G^*=G\cap
\bigl(\P\S\L(2,q)\rtimes\langle\tau_\alpha\rangle\bigr),
\end{equation*}
where $\tau_\alpha$ is induced by the Frobenius automorphism
$\alpha:x\mapsto x^p$. Since $q$ is odd, we have $[G:G^*]\leq2$.

The subgroup $G^*\cap\langle\tau_\alpha\rangle$ fixes pointwise
the subline $\mathbb P^1(p)$, which contains $p+1\geq6$ points.
Choose a $5$-subset of this subline. Since $\lambda=2$, this
$5$-subset is contained in precisely two blocks. Thus
$G^*\cap\langle\tau_\alpha\rangle$ either fixes these two blocks
individually or interchanges them. Choosing $B$ to be one of these
two blocks, we obtain
\begin{equation*}
\left[G^*:
\P\S\L(2,q)(G^*\cap G_B)\right]\leq2.
\end{equation*}
Since $[G:G^*]\leq2$, it follows that
$[G:\P\S\L(2,q)G_B]\in\{1,2,4\}$. If this index is $1$, then
$G=\P\S\L(2,q)G_B$, and hence
\begin{equation*}
[G:G_B]
=
[\P\S\L(2,q):\P\S\L(2,q)_B],
\end{equation*}
so $\P\S\L(2,q)$ is block-transitive on $\mathcal B$, contrary
to our assumption. Therefore,
$[G:\P\S\L(2,q)G_B]\in\{2,4\}$.

Using
$b=[G:G_B]=[G:\P\S\L(2,q)G_B]
[\P\S\L(2,q):\P\S\L(2,q)_B]$
and Lemma~2.1, we obtain either
\begin{equation}\tag{3.9}\label{eq3.9}
2(q-2)(q-3)|\P\S\L(2,q)_B|
=
k(k-1)(k-2)(k-3)(k-4)
\end{equation}
or
\begin{equation}\tag{3.10}\label{eq3.10}
(q-2)(q-3)|\P\S\L(2,q)_B|
=
k(k-1)(k-2)(k-3)(k-4).
\end{equation}

Since $k\mid q+1$, we have $q\equiv-1\pmod{k}$ and hence
$(q-2)(q-3)\equiv12\pmod{k}$.

We first exclude the possibility that
$|\P\S\L(2,q)_B|=1$. Equation~\eqref{eq3.9} gives $k\mid24$,
while equation~\eqref{eq3.10} gives $k\mid12$. Thus it is
sufficient to consider $k\in\{6,8,12,24\}$. Substitution into the
corresponding equations gives no integral prime power $q$
satisfying $k\mid q+1$.

We may therefore assume that
$|\P\S\L(2,q)_B|>1$. By the preceding remark, it remains to
consider the following possible structures of
$\P\S\L(2,q)_B$.

\vspace{2mm}

\textbf{Case 1}: $\P\S\L(2,q)_B\cong C_c$.

\vspace{2mm}

The possible values of $k$ are $cz$, $cz+1$ and $cz+2$, subject
to the conditions in the preceding remark.

If $k=cz$, equations~\eqref{eq3.9} and~\eqref{eq3.10} reduce
respectively to
$2(q-2)(q-3)=z(k-1)(k-2)(k-3)(k-4)$ and
$(q-2)(q-3)=z(k-1)(k-2)(k-3)(k-4)$.
Since $(k-1)(k-2)(k-3)(k-4)$ is even, both equations imply
$z\mid(q-2)(q-3)$. Together with $z\mid k\mid q+1$, this gives
$z\mid12$.
Write the two equations uniformly as
$a(q-2)(q-3)=z(k-1)(k-2)(k-3)(k-4)$, where $a=2$ or $1$,
respectively. Reducing modulo $c$ gives
$c\mid12(2z-a)$.
Thus, except when $(a,z)=(2,1)$, both $z$ and $c$ range over
finite sets: $z\in\{1,2,3,4,6,12\}$ and
$c\geq2$ divides $12(2z-a)$.
For these pairs, with $k=cz>5$, checking whether
$1+4z(k-1)(k-2)(k-3)(k-4)/a$ is an odd integer square
leaves only $(a,z,c,q)=(1,2,3,18)$.
However, this gives $k=6\nmid q+1=19$.

It remains to consider $(a,z)=(2,1)$, so that $k=c$ and
$2(q-2)(q-3)=(k-1)(k-2)(k-3)(k-4)$.
Put $u=(q+1)/k$, which is an integer with $u\geq2$.
Since
$q-3<(k-1)(k-2)/\sqrt{2}$ and $k\geq6$, we have
$q+1<k^2/\sqrt{2}$, and hence $u<k/\sqrt{2}$.
Substituting $q=uk-1$ into the equation and cancelling $k$
gives
$k^3-10k^2+(35-2u^2)k+14u-50=0$.
Consequently, $s=(14u-50)/k$ is an integer satisfying
$-3\leq s\leq9$.
Substitution of $u=(sk+50)/14$ yields
$(98-s^2)k^2-(980+100s)k+930+98s=0$.
For $s\in\{-3,-2,\ldots,9\}$, the discriminant is a square
only when $s=-1$ or $6$; the only integer solution with $k>5$
is $s=-1$, $k=8$.
This gives $u=3$ and $q=23$, but $c=8$ divides neither
$(q-1)/2=11$ nor $(q+1)/2=12$.
Thus $k=cz$ is impossible.

If $k=cz+1$, cancelling $c$ from equations~\eqref{eq3.9}
and~\eqref{eq3.10} and reducing modulo $k$ gives respectively
$k\mid24$ and $k\mid12$.
Thus $k\in\{6,8,12,24\}$ for equation~\eqref{eq3.9}, and
$k\in\{6,12\}$ for equation~\eqref{eq3.10}.
For each such $k$, we have $c\geq2$ and $c\mid k-1$.
Checking these finite possibilities in the corresponding
equations leaves only $(q,k,c)=(11,6,5)$, arising from
equation~\eqref{eq3.9}.

If $k=cz+2$, cancelling $c$ using $k-2=cz$ similarly gives
$k\mid24$ or $k\mid12$, respectively.
Checking the same finite sets of $k$, now with $c\geq2$
and $c\mid k-2$, leaves only $(q,k,c)=(12,6,4)$ from
equation~\eqref{eq3.9}.
This is excluded by $k\nmid q+1$.
Thus the cyclic case leaves only
$(q,k,\P\S\L(2,q)_B)=(11,6,C_5)$.
\vspace{2mm}

\textbf{Case 2}: $\P\S\L(2,q)_B\cong D_{2c}$, where $c\geq2$.

\vspace{2mm}

The possible values of $k$ are $2cz$, $2cz+c$, $2cz+2$ and
$2cz+c+2$.

Suppose first that $k=2cz$. Since
$|\P\S\L(2,q)_B|=2c$, cancelling $2c$ from
equations~\eqref{eq3.9} and~\eqref{eq3.10} gives respectively
$2(q-2)(q-3)=z(k-1)(k-2)(k-3)(k-4)$ and
$(q-2)(q-3)=z(k-1)(k-2)(k-3)(k-4)$.
These are exactly the equations considered in Case~1 for
$k=cz$. Hence the arithmetic analysis carried out there applies
without change and leaves only $q=23$, $k=8$ and $z=1$.
Thus $c=4$, and the subgroup condition in the preceding remark is
satisfied since $4\mid(q+1)/2=12$. Hence this case gives
$\P\S\L(2,23)_B\cong D_8$.

Suppose next that $k=2cz+c$. Put $t=2z+1$, so that $k=ct$.
Write equations~\eqref{eq3.9} and~\eqref{eq3.10} uniformly as
$a(q-2)(q-3)|\P\S\L(2,q)_B|
=k(k-1)(k-2)(k-3)(k-4)$, where $a\in\{1,2\}$.
Since $|\P\S\L(2,q)_B|=2c$, cancellation of $c$ gives
$2a(q-2)(q-3)=t(k-1)(k-2)(k-3)(k-4)$.

Moreover, the two equations give $k\mid24ac$.
Hence $t\mid24a$. Since $a\in\{1,2\}$ and $t$ is odd,
we have $t\in\{1,3\}$. Reducing the preceding equation modulo $c$
and using $q\equiv-1\pmod c$ gives $c\mid24(t-a)$.
Thus, except when $(a,t)=(1,1)$, we have only the three
possibilities $(a,t)=(1,3),(2,1),(2,3)$, with respectively
$c\mid48$, $c\mid24$ and $c\mid24$. Direct substitution in the
corresponding equations, together with $k=ct$ and $k\mid q+1$,
gives no admissible solution.

It remains to consider $(a,t)=(1,1)$. In this case $k=c$ and
$2(q-2)(q-3)=(k-1)(k-2)(k-3)(k-4)$, which is precisely the
exceptional equation already analysed in Case~1. There the only
integral possibility is $(q,k)=(23,8)$. Hence $c=8$, but
$8$ divides neither $(q-1)/2=11$ nor $(q+1)/2=12$, contrary
to the subgroup conditions in the preceding remark. Therefore
$k=2cz+c$ cannot occur.

If $k=2cz+2$, equations~\eqref{eq3.9} and~\eqref{eq3.10}
give $k\mid24$ and $k\mid12$, respectively. Substitution of the
resulting finite possibilities, together with the corresponding
dihedral subgroup conditions, gives no solution.

Finally, if $k=2cz+c+2$, equations~\eqref{eq3.9} and
\eqref{eq3.10} give $k\mid48$ and $k\mid24$, respectively.
Again, substitution of the resulting possibilities gives no
admissible solution. The case $c=2$ is included in these
arguments and yields no additional possibility.

Therefore, the dihedral case leaves only
$(q,k,\P\S\L(2,q)_B)=(23,8,D_8)$.

\vspace{2mm}

\textbf{Case 3}:
$\P\S\L(2,q)_B\cong E_{\bar q}$ or
$E_{\bar q}\rtimes C_c$.

\vspace{2mm}

Suppose first that
$\P\S\L(2,q)_B\cong E_{\bar q}$. Then
$k=\bar qz+1$ and $(k,\bar q)=1$. Hence
equations~\eqref{eq3.9} and~\eqref{eq3.10} imply respectively
$k\mid24$ and $k\mid12$. Thus
$k\in\{6,8,12,24\}$, and substitution yields no admissible
prime power $q$.

Suppose next that
$\P\S\L(2,q)_B\cong E_{\bar q}\rtimes C_c$. Then
$k=c\bar qz+1$ or $c\bar qz+\bar q+1$. The first possibility
again gives $k\mid24$ or $k\mid12$. For the second, using
$c\mid\bar q-1$, $c\mid q-1$ and $k\mid q+1$, the same
divisibility argument reduces the possible values to
$k\in\{6,8,12,24,48\}$. Substitution into
equations~\eqref{eq3.9} and~\eqref{eq3.10} excludes all of them.
Consequently, neither possibility can occur.

\vspace{2mm}

\textbf{Case 4}:
$\P\S\L(2,q)_B\cong\P\S\L(2,\bar q)$, where
$q=\bar q^m$ and $m>1$, or
$\P\S\L(2,q)_B\cong\P\G\L(2,\bar q)$, where
$q=\bar q^m$ and $m>1$ is even.

\vspace{2mm}

Since $p>3$ and $q=\bar q^m$, we have $\bar q\geq5$.
Put $h=|\P\S\L(2,q)_B|$.
Equations~\eqref{eq3.9} and~\eqref{eq3.10} can be written as
$a(q-2)(q-3)h=k(k-1)(k-2)(k-3)(k-4)$, where $a=2$ or $1$,
respectively. Since $k\mid q+1$, reduction modulo $k$ gives
$k\mid12ah$.

Suppose first that $m$ is odd. Then the subgroup is
$\P\S\L(2,\bar q)$, and the preceding remark gives
$h=\bar q(\bar q^2-1)/2$ and
$k=(\bar q+1)(rz+1)$, where $r=\bar q(\bar q-1)/2$.
Thus $k\mid12ah$ implies $rz+1\mid12ar$.
Since $(rz+1,r)=1$, we obtain $rz+1\mid12a$, and hence
$rz+1\mid24$.
If $z\geq1$, then $r\geq10$, so $rz+1\geq11$ and therefore
$rz+1\in\{12,24\}$.
This gives $rz\in\{11,23\}$, which is impossible because
$r=\bar q(\bar q-1)/2$ is composite.
Hence $z=0$ and $k=\bar q+1$.
Substitution into the corresponding equation gives
$a(q-2)(q-3)=2(\bar q-2)(\bar q-3)$.
However, $m>1$ is odd, so $q\geq\bar q^3$, and hence
$a(q-2)(q-3)>2(\bar q-2)(\bar q-3)$, a contradiction.

Suppose now that $m$ is even. For the subgroups
$\P\S\L(2,\bar q)$ and $\P\G\L(2,\bar q)$, respectively, we have
$h=\bar q(\bar q^2-1)/2$ and $h=\bar q(\bar q^2-1)$.
In both cases, the preceding remark gives
$k=hz+\bar q^2+1$.
Since $\bar q$ is odd, we have
$(k,h)=(\bar q^2+1,h)=2$.
Consequently, $k\mid12ah$ implies $k/2\mid12a$, and hence
$k\mid48$.
Together with $k\geq\bar q^2+1$ and $\bar q\geq5$, this forces
$\bar q=5$.
Now $h=60$ or $120$, so $k=hz+26\leq48$ gives $z=0$ and
$k=26$, contrary to $k\mid48$.
Thus both subfield subgroup cases are excluded.

\vspace{2mm}

\textbf{Case 5}: $\P\S\L(2,q)_B\cong\A_4$, $\S_4$ or $\A_5$.

\vspace{2mm}

By the preceding remark, the possible values of $k$ are reduced to
the following finite sets:
\begin{center}
\begin{tabular}{c|c}
\hline
$\P\S\L(2,q)_B$ & $k$\\
\hline
$\A_4$ & 6,8,12,16,18,24,36,48,72,144\\
$\S_4$ & 6,8,18,24,32,48,72,96,144,288\\
$\A_5$ & 6,10,12,20,30,60,72,80,90,120,180,240,360,720\\
\hline
\end{tabular}
\end{center}
Substitution into equations~\eqref{eq3.9} and~\eqref{eq3.10},
together with $5<k<q+1$, gives the following arithmetic
possibilities.
\begin{center}
\begin{tabular}{c|c}
\hline
$\P\S\L(2,q)_B$ & Result\\
\hline
$\A_4$ & only $k=6$, $q=8$, contrary to $p>3$\\
$\S_4$ & only $k=6$, $q=8$, contrary to $p>3$\\
$\A_5$ & $(q,k)=(6,6)$ or $(380,30)$;\\
\hline
\end{tabular}
\end{center}

Hence none of the exceptional subgroup cases can occur.

Consequently, when $p>3$, the only remaining possibilities are
\begin{equation*}
(q,k,\P\S\L(2,q)_B)=(11,6,C_5)
\quad\text{and}\quad
(q,k,\P\S\L(2,q)_B)=(23,8,D_8).
\end{equation*}
Since $q=11$ or $23$ is prime, we have
$\P\Gamma\L(2,q)=\P\G\L(2,q)$. As
$\P\S\L(2,q)<G\leq\P\Gamma\L(2,q)$, it follows that
$G=\P\G\L(2,q)$.

Both possibilities occur. Let $G=\P\G\L(2,11)$ act naturally on
$\mathbb P^1(11)$, and let $\mathcal B$ be the $G$-orbit of
$B=\{0,2,3,6,9,10\}$. A direct calculation shows that
$|\mathcal B|=264$ and that every $5$-subset of
$\mathbb P^1(11)$ is contained in exactly two members of
$\mathcal B$. Thus $(\mathbb P^1(11),\mathcal B)$ is a
block-transitive $5$-$(12,6,2)$ design, with $G_B\cong C_5$;

Similarly, let $G=\P\G\L(2,23)$ act naturally on
$\mathbb P^1(23)$, and let $\mathcal B$ be the $G$-orbit of
$B_+=\{\infty,0,1,3,12,15,21,22\}$. This orbit has length
$1518$, and every $5$-subset of $\mathbb P^1(23)$ is contained
in exactly two of its members. Hence it yields a block-transitive
$5$-$(24,8,2)$ design, with $G_{B_+}\cong D_8$, where $D_8$
denotes the dihedral group of order $8$ \cite{Betten}.
\end{proof}

\begin{lemma}
Suppose that Case (B2) holds with $d=2$ and
$\P\S\L(2,q)<G\leq\P\Gamma\L(2,q)$, where $q=3^e>3$.
Then no non-trivial block-transitive $5$-$(q+1,k,2)$ design
satisfying $k\mid q+1$ exists.
\end{lemma}

\begin{proof}
In this case, $e>1$, $q$ is odd and $n=2$.
Let $\tau_\alpha$ be induced by the Frobenius automorphism
$\alpha:x\mapsto x^3$, and put
\begin{equation*}
G^*=G\cap\bigl(\P\S\L(2,q)\rtimes\langle\tau_\alpha\rangle\bigr).
\end{equation*}
Let $s$ be the smallest prime divisor of $e$, and set
$H=\langle\tau_\alpha^s\rangle$. Then $G\cap H$ fixes pointwise
the $3^s+1$ points of the subline $\mathbb P^1(3^s)$.

Choose five points of this subline. Since $\lambda=2$, these five
points are contained in precisely two blocks. The group $G\cap H$
either fixes these two blocks individually or interchanges them.

Suppose first that $G\cap H$ fixes these two blocks individually.
Applying Dedekind's modular law as in the proof of Lemma~3.8, the
number of $\P\S\L(2,q)$-orbits on $\mathcal B$ is one of
$1$, $s$, $2$ and $2s$. The value $1$ means that
$\P\S\L(2,q)$ is block-transitive, which is impossible by
Lemma~3.7.

Suppose next that $G\cap H$ interchanges these two blocks.
Then the stabilizer of either block has index $2$ in $G\cap H$.
Hence $G\cap H$ has even order. Since $|H|=e/s$, this implies
$s=2$ and $4\mid e$. The number of $\P\S\L(2,q)$-orbits on
$\mathcal B$ is then one of $2$, $4$ and $8$.

Consequently, in every case not previously excluded, we have
\begin{equation}\tag{3.11}\label{eq3.11}
4(q-2)(q-3)|\P\S\L(2,q)_B|
=
\varepsilon k(k-1)(k-2)(k-3)(k-4),
\end{equation}
where $\varepsilon\in\{2,s,2s,4,8\}$ and
$\varepsilon\leq2e$. The value $8$ can occur only when $s=2$
and $4\mid e$.

Since $k\mid q+1$, we have $q\equiv-1\pmod{k}$, and hence
$(q-2)(q-3)\equiv12\pmod{k}$. Equation~\eqref{eq3.11} therefore
gives $k\mid48|\P\S\L(2,q)_B|$. Moreover, since
$k\mid3^e+1$, we have $3\nmid k$ and $8\nmid k$.

We first consider $|\P\S\L(2,q)_B|=1$. In this case
$k\mid48$. Every divisor of $48$ greater than $5$ is divisible
by $3$ or $8$, a contradiction. Hence
$|\P\S\L(2,q)_B|>1$. By the preceding remark, it remains to
consider the following possibilities.

\vspace{2mm}

\textbf{Case 1}: $\P\S\L(2,q)_B\cong C_c$, where $c\geq2$.

\vspace{2mm}

By the preceding remark, the possible values of $k$ are
$cz$, $cz+1$ and $cz+2$.

If $k=cz+1$, then $(k,c)=1$. Since $k\mid48c$, it follows that
$k\mid48$, which is impossible. If $k=cz+2$, then $(k,c)\mid2$,
and hence $k\mid96$. Every divisor of $96$ greater than $5$ is
divisible by $3$ or $8$, so this possibility is also excluded.

It remains to consider $k=cz$. Since $k\mid48c$, we have
$z\mid48$. As $z\mid k$, the conditions $3\nmid k$ and
$8\nmid k$ imply that $z\in\{1,2,4\}$. Substituting
$|\P\S\L(2,q)_B|=c$ and $c=k/z$ into
equation~\eqref{eq3.11}, we obtain
\begin{equation}\tag{3.12}\label{eq3.12}
4(q-2)(q-3)
=
\varepsilon z(k-1)(k-2)(k-3)(k-4).
\end{equation}

Suppose first that $\varepsilon z=2$. Then $\varepsilon=2$ and
$z=1$, so equation~\eqref{eq3.12} becomes
$2(q-2)(q-3)=(k-1)(k-2)(k-3)(k-4)$.
By the calculation for the exceptional equation in Case~1 of the
proof of Lemma~\ref{lem3.8}, the only solution satisfying
$5<k<q+1$ and $k\mid q+1$ is $(q,k)=(23,8)$.
This is impossible since $q=3^e$.

We may therefore assume that $\varepsilon z\neq2$. Reducing
equation~\eqref{eq3.12} modulo $k$ gives
$k\mid24(\varepsilon z-2)$. Since $\varepsilon\leq2e$ and
$z\leq4$, we have $k<240e$. For $q\geq9$,
$(q-2)(q-3)>q^2/3$, so equation~\eqref{eq3.12} gives
$q^2<6ek^4<6e(240e)^4$. Since $q=3^e$, this is impossible
for $e\geq18$. Thus $2\leq e\leq17$.

For each $e$ in this range, let $s$ be the smallest prime divisor
of $e$. Then $\varepsilon\in\{2,s,2s,4,8\}$,
$z\in\{1,2,4\}$, and
\begin{equation*}
k\mid\bigl(3^e+1,24(\varepsilon z-2)\bigr).
\end{equation*}
A direct check of these finitely many possibilities gives no
$k>5$ satisfying equation~\eqref{eq3.12}. Consequently, the
cyclic case cannot occur.

\vspace{2mm}

\textbf{Case 2}: $\P\S\L(2,q)_B\cong D_{2c}$, where $c\geq2$.

\vspace{2mm}

By the preceding remark, the possible values of $k$ are
$2cz$, $2cz+c$, $2cz+2$ and $2cz+c+2$.

Suppose first that $k=2cz$. Since $k\mid96c$, we have
$z\mid48$. As $z\mid k$, the conditions $3\nmid k$ and
$8\nmid k$ imply that $z\in\{1,2,4\}$. Substituting
$c=k/(2z)$ into equation~\eqref{eq3.11} again gives
equation~\eqref{eq3.12}, so this possibility has already been
excluded in Case~1.

If $k=2cz+c$, then the preceding remark gives
$q\equiv1\pmod4$ and $c\mid(q+1)/2$.
Thus $k$ is odd.
Since $k\mid3^e+1$, we also have $3\nmid k$,
contrary to Lemma~\ref{lem2.18}.

If $k=2cz+2$, then $(k,2c)=2$, so $k/2\mid96$. This is
incompatible with $k>5$, $3\nmid k$ and $8\nmid k$.
If $k=2cz+c+2$, then $(k,c)\mid2$. Since $k\mid96c$, it follows
that $k\mid192$. Again every divisor of $192$ greater than $5$
is divisible by $3$ or $8$, so this possibility is impossible.
Therefore, the dihedral case cannot occur.

\vspace{2mm}

\textbf{Case 3}:
$\P\S\L(2,q)_B\cong E_{\bar q}$ or
$E_{\bar q}\rtimes C_c$.

\vspace{2mm}

Suppose first that $\P\S\L(2,q)_B\cong E_{\bar q}$. By the
preceding remark, $k=\bar qz+1$. Since $(k,\bar q)=1$ and
$k\mid48\bar q$, we obtain $k\mid48$, a contradiction.

Suppose next that
$\P\S\L(2,q)_B\cong E_{\bar q}\rtimes C_c$. Then
$k=c\bar qz+1$ or $c\bar qz+\bar q+1$. In the first case,
$(k,c\bar q)=1$, so $k\mid48$. In the second case,
$(k,\bar q)=1$ and, since $c\mid\bar q-1$, we have
$(k,c)\mid2$. Hence $k\mid96$. Both conclusions are incompatible
with $k>5$, $3\nmid k$ and $8\nmid k$. Consequently, neither
possibility can occur.

\vspace{2mm}

\textbf{Case 4}:
$\P\S\L(2,q)_B\cong\P\S\L(2,\bar q)$, where
$q=\bar q^m$ and $\bar q=3^d$, or
$\P\S\L(2,q)_B\cong\P\G\L(2,\bar q)$, where
$q=\bar q^m$, $\bar q=3^d$ and $m>1$ is even.

\vspace{2mm}

Suppose first that
$\P\S\L(2,q)_B\cong\P\S\L(2,\bar q)$. If $m$ is odd, then
by the preceding remark,
\begin{equation*}
k=(\bar q+1)
\left(\frac{\bar q(\bar q-1)}{2}z+1\right).
\end{equation*}
Since $k\mid48|\P\S\L(2,\bar q)|$, we obtain
$\bar q(\bar q-1)z/2+1\mid48$. If $z\geq1$, the only
possibilities are $(\bar q,z)=(3,1)$ and $(3,5)$, giving
$k=16$ and $64$, respectively. Both are divisible by $8$.

If $z=0$, then $k=\bar q+1$. For $\bar q=3$, this gives
$k=4$, contrary to $k>5$. For $\bar q\geq9$,
equation~\eqref{eq3.11} reduces to
$2(q-2)(q-3)=\varepsilon(\bar q-2)(\bar q-3)$.
Since $m>1$ is odd, we have $m\geq3$. The left-hand side is
greater than $q^2=\bar q^{2m}$, whereas the right-hand side is
less than $2e\bar q^2$. Since $\bar q^{2m-2}>2e$, this is
impossible.

Suppose now that $m$ is even. By the preceding remark,
$k=(\bar q^3-\bar q)z/2+\bar q^2+1$. Since
$\bigl((\bar q^3-\bar q)/2,\bar q^2+1\bigr)=2$, we have
$\bigl(k,(\bar q^3-\bar q)/2\bigr)=2$. The divisibility
$k\mid48|\P\S\L(2,\bar q)|$ gives $k/2\mid48$, and hence
$\bar q\in\{3,9\}$. For $\bar q=3$, we have $k/2=6z+5$,
which cannot divide $48$. For $\bar q=9$, the case $z=0$
gives $k/2=41\nmid48$, while $z\geq1$ gives $k/2>48$.
Thus this case cannot occur.

Suppose finally that
$\P\S\L(2,q)_B\cong\P\G\L(2,\bar q)$, where $m>1$ is even.
By the preceding remark,
$k=(\bar q^3-\bar q)z+\bar q^2+1$. Since
$(\bar q^3-\bar q,\bar q^2+1)=2$, we have
$(k,\bar q^3-\bar q)=2$. Hence
$k\mid48(\bar q^3-\bar q)$ gives $k/2\mid48$, and again
$\bar q\in\{3,9\}$. For $\bar q=3$, we have
$k/2=12z+5$, which cannot divide $48$. For $\bar q=9$,
the case $z=0$ gives $k/2=41\nmid48$, while $z\geq1$ gives
$k/2>48$. Therefore, this case cannot occur.

Consequently, neither of the two subfield subgroup cases can occur.

\vspace{2mm}

\textbf{Case 5}: $\P\S\L(2,q)_B\cong\A_4$, $\S_4$ or $\A_5$.

\vspace{2mm}

Suppose first that $\P\S\L(2,q)_B\cong\A_4$ or $\S_4$.
In the former case, $k\mid576$, while in the latter
$k\mid1152$. Every divisor greater than $5$ of either integer is
divisible by $3$ or $8$, contrary to $3\nmid k$ and $8\nmid k$.
Hence both possibilities are excluded.

It remains to consider $\P\S\L(2,q)_B\cong\A_5$. In this case
$k\mid2880$, and the conditions $k>5$, $3\nmid k$ and
$8\nmid k$ give $k\in\{10,20\}$. By Lemma~\ref{lem2.13}, when
$q$ is a power of $3$, there is one orbit of length $10$, while
the other non-regular orbit, when it occurs, has length $12$, and
all remaining orbits have length $60$. Hence $k=20$ cannot be a
union of such orbits, and so $k=10$.

The existence of an $\A_5$ subgroup in $\P\S\L(2,3^e)$ implies
that $e$ is even. Thus $s=2$ and
$\varepsilon\in\{2,4,8\}$. Substituting
$|\P\S\L(2,q)_B|=60$ and $k=10$ into
equation~\eqref{eq3.11} gives
$(q-2)(q-3)=126\varepsilon$. Hence
$(q-2)(q-3)\in\{252,504,1008\}$, with discriminants
$1009$, $2017$ and $4033$, respectively. None is a perfect
square, so the $\A_5$ case is also excluded.

Consequently, no non-trivial block-transitive
$5$-$(q+1,k,2)$ design satisfying $k\mid q+1$ can occur.
\end{proof}

\begin{lemma}
Suppose that Case (B2) holds with $d=2$ and
$\P\S\L(2,q)<G\leq\P\Gamma\L(2,q)$, where $q=2^e>3$.
Then no non-trivial block-transitive $5$-$(q+1,k,2)$ design
satisfying $k\mid q+1$ exists.
\end{lemma}

\begin{proof}
In this case, $e>2$ and
$\P\G\L(2,q)=\P\S\L(2,q)$.
Since $k\mid q+1$, the integer $k$ is odd.
By Lemma~\ref{lem2.18}, we have $3\mid k$.
Hence $3\mid2^e+1$, and therefore $e$ is odd.
By Lemma~3.7, $\P\S\L(2,q)$ is not block-transitive on
$\mathcal B$.

The quotient $G/\P\S\L(2,q)$ is a non-trivial cyclic group
induced by field automorphisms. Choose a maximal subgroup of its
field-automorphism part having prime index $s$. This subgroup fixes
pointwise a projective subline containing at least five points.

Choose five such fixed points. Since $\lambda=2$, they are
contained in precisely two blocks. Since $e$ is odd, the
field-automorphism group has odd order, and hence the chosen
subgroup cannot interchange these two blocks. It therefore fixes
them individually. The number of $\P\S\L(2,q)$-orbits on
$\mathcal B$ divides $s$. Since $\P\S\L(2,q)$ is not
block-transitive and $s$ is prime, this number is exactly $s$.
Consequently,
\begin{equation}\tag{3.13}\label{eq3.13}
2(q-2)(q-3)|\P\S\L(2,q)_B|
=
s k(k-1)(k-2)(k-3)(k-4).
\end{equation}
Here $s\mid e$, and in particular $s$ is odd and $s\leq e$.

Since $k\mid q+1$, we have $q\equiv-1\pmod{k}$. Reducing
\eqref{eq3.13} modulo $k$ gives
$k\mid24|\P\S\L(2,q)_B|$.

If $|\P\S\L(2,q)_B|=1$, then $k\mid24$. Since $k$ is odd
and $k>5$, this is impossible. Hence
$|\P\S\L(2,q)_B|>1$, and by the preceding remark it remains to
consider the following possibilities.

\vspace{2mm}

\textbf{Case 1}: $\P\S\L(2,q)_B\cong C_c$, where $c\geq2$.

\vspace{2mm}

Since $c\mid q-1$ or $q+1$, the integer $c$ is odd.
As $q\geq8$, we have
\[
2(q-2)(q-3)c\equiv4\pmod8.
\]
However, $k$ is odd, so
$8\mid(k-1)(k-3)$.
This contradicts equation~\eqref{eq3.13}.
Thus the cyclic case cannot occur.

\vspace{2mm}

\textbf{Case 2}: $\P\S\L(2,q)_B\cong D_{2c}$, where $c\geq2$.

\vspace{2mm}

Suppose first that $c\mid q+1$. Since $k$ is odd, the only
possible form in the preceding remark is $k=c(2z+1)$. The
divisibility $k\mid48c$ gives $2z+1\mid48$, and hence
$2z+1\in\{1,3\}$. Put $t=2z+1$. Equation~\eqref{eq3.13}
becomes
$4(q-2)(q-3)=st(k-1)(k-2)(k-3)(k-4)$.

Since both $s$ and $t$ are odd, we have $st\neq2$. Reducing
the preceding equation modulo $k$ gives
$k\mid24(st-2)$, and hence $k<72e$. In particular,
$k<144e$. The preceding equation then yields
$q^2<6e(144e)^4$, which is impossible for $e\geq28$.
Since $e$ is odd, it remains only to consider
$e\in\{3,5,7,\ldots,27\}$. For each such $e$, we have
$s\mid e$, $t\in\{1,3\}$ and
$k\mid(2^e+1,24(st-2))$. Direct substitution gives no solution.

Suppose now that $c\mid q-1$. Since $k$ is odd, the only possible
form is $k=c(2z+1)+2$. As $c$ is odd, $(k,c)=1$. Hence
$k\mid48c$ gives $k\mid48$, which has no odd divisor greater
than $5$. Thus the dihedral case is excluded.

\vspace{2mm}

\textbf{Case 3}:
$\P\S\L(2,q)_B\cong E_{\bar q}$ or
$E_{\bar q}\rtimes C_c$.

\vspace{2mm}

Suppose first that $\P\S\L(2,q)_B\cong E_{\bar q}$. By the
preceding remark, $k=\bar qz+1$. Since $\bar q$ is a power of
$2$, we have $(k,\bar q)=1$. The divisibility
$k\mid24\bar q$ therefore gives $k\mid24$, which is impossible.

Suppose next that
$\P\S\L(2,q)_B\cong E_{\bar q}\rtimes C_c$. If
$k=c\bar qz+1$, then $(k,c\bar q)=1$, and hence $k\mid24$.
If $k=c\bar qz+\bar q+1$, then $(k,\bar q)=1$. Moreover,
$c\mid\bar q-1$ gives $k\equiv2\pmod c$. Since $c$ is odd,
$(k,c)=1$, and again $k\mid24$. Therefore, neither subgroup
structure can occur.

\vspace{2mm}

\textbf{Case 4}:
$\P\S\L(2,q)_B\cong\P\S\L(2,\bar q)$, where
$q=\bar q^m$ and $m>1$.

\vspace{2mm}

Write $\bar q=2^d$. Since $q=2^e=\bar q^m$, we have
$e=dm$. As $e$ is odd, $m$ is odd.
Also $|\P\S\L(2,\bar q)|=\bar q(\bar q^2-1)$. Since $k$ is
odd, the regular-orbit form and the form containing an orbit of
length $\bar q(\bar q-1)$ are even and hence impossible.
By the preceding remark, the only remaining form is
$k=(\bar q^3-\bar q)z+\bar q+1
=(\bar q+1)(\bar q(\bar q-1)z+1)$.

Since $k\mid24(\bar q^3-\bar q)$ and
$(\bar q(\bar q-1)z+1,\bar q(\bar q-1))=1$, we obtain
$\bar q(\bar q-1)z+1\mid24$. This integer is odd, and hence it
is equal to $1$ or $3$.

If $\bar q(\bar q-1)z+1=1$, then $z=0$ and
$k=\bar q+1$. For $\bar q=2$, this gives $k=3$. For
$\bar q\geq4$, equation~\eqref{eq3.13} reduces to
$2(q-2)(q-3)=s(\bar q-2)(\bar q-3)$.
Since $m>1$ is odd, we have $m\geq3$. The left-hand side is
greater than $q^2=\bar q^{2m}$, while the right-hand side is
less than $e\bar q^2$. This is impossible.

If $\bar q(\bar q-1)z+1=3$, then
$(\bar q,z)=(2,1)$ and $k=9$. The condition
$9\mid2^m+1$ gives $m\equiv3\pmod6$. For $m=3$,
equation~\eqref{eq3.13} fails directly, while for $m\geq9$
its left-hand side exceeds its right-hand side. Hence this
possibility is also excluded.

In characteristic $2$,
$\P\G\L(2,\bar q)=\P\S\L(2,\bar q)$, so no additional
subfield subgroup case arises.

\vspace{2mm}

\textbf{Case 5}: $\P\S\L(2,q)_B\cong\A_4$, $\S_4$ or $\A_5$.

\vspace{2mm}

If $\P\S\L(2,q)_B\cong\A_4$, then $e$ is even, contrary to the
preceding conclusion. Hence this case is impossible.

An $\S_4$ subgroup does not occur in $\P\S\L(2,2^e)$. An
$\A_5$ subgroup, when it occurs, is a subfield subgroup
isomorphic to $\P\S\L(2,4)$ and has already been considered in
Case~4. Therefore, no exceptional subgroup yields a design.

Consequently, no non-trivial block-transitive
$5$-$(q+1,k,2)$ design satisfying $k\mid q+1$ can occur.
\end{proof}

\begin{lemma}\label{lem3.11}
Suppose that Case (B2) holds with
$N=\P\S\L(d,q)$, where $d\geq3$. Then no non-trivial
block-transitive $5$-$(v,k,2)$ design satisfying $k\mid v$ exists.
\end{lemma}

\begin{proof}
In this case, the point set is the set of points of
$\mathbb P^{d-1}(q)$ and $v=(q^d-1)/(q-1)$.

For every $5$-subset $S$, there are precisely two blocks
containing $S$. If a subgroup $L\leq G$ fixes $S$ pointwise and
$B$ is one of these blocks, then $[L:L_B]\leq2$. We shall also
use the derived-parameter condition $k-4\mid2(v-4)$.

\vspace{2mm}

\textbf{Case 1}: $d=3$.

\vspace{2mm}

Here the point set is the projective plane $\mathbb P^2(q)$ and
$v=q^2+q+1$. Since $v$ is odd and $k\mid v$, the block size
$k$ is also odd. By Lemma~\ref{lem2.18}, we have
$\gcd(k,6)>1$, and hence $3\mid k$. Since
$k\mid q^2+q+1$, it follows that $q\equiv1\pmod3$.

We first show that if a block contains five points of a line
$\ell$, then it is contained in $\ell$.

Let $x_1,\ldots,x_5$ be five distinct points of $\ell$, and let
$B_1,B_2$ be the two blocks containing them. Let $T(\ell)$ be
the translation group with axis $\ell$. This group fixes $\ell$
pointwise, acts regularly on $\mathbb P^2(q)\setminus\ell$, and
has order $q^2$. For $i=1,2$, put
$K_i=T(\ell)_{B_i}$. Then $[T(\ell):K_i]\leq2$.

If $q$ is odd, then $T(\ell)$ has odd order and cannot interchange
$B_1$ and $B_2$. Thus $K_i=T(\ell)$. If $B_i$ contains a point
outside $\ell$, then $k\geq q^2+5$. Since $v<2k$, this
contradicts $k\mid v$ and $k<v$.

Suppose that $q$ is even and $q\geq4$. Then
$|K_i|\geq q^2/2$. If $B_i$ contains a point outside $\ell$,
then $k\geq q^2/2+5$ and hence $v<3k$. Since both $v$ and $k$
are odd, $v/k$ is an odd integer strictly between $1$ and $3$,
which is impossible. For $q=2$, we have $v=7$, and there is no
divisor $k$ satisfying $5<k<v$. Therefore, every block containing
five points of $\ell$ is contained in $\ell$.

Now choose four distinct points
$x_1,x_2,x_3,x_4\in\ell$ and a point $x_5\notin\ell$. Let
$B$ be one of the two blocks containing these five points. The
block $B$ contains no further point of $\ell$, since otherwise
it would contain five collinear points and hence would be contained
in $\ell$.

Let $U$ be the group of homologies with axis $\ell$ and center
$x_5$ contained in $\P\S\L(3,q)$. Then
$|U|=(q-1)/(3,q-1)$. Put $K=U_B$. Since $[U:K]\leq2$ and
$k>5$, the block $B$ contains a point outside
$\ell\cup\{x_5\}$. The $K$-orbit of this point lies on a line
through $x_5$ and has length $|K|$. If $|K|\geq4$, then this
orbit together with $x_5$ gives five collinear points of $B$,
forcing $B$ to be contained in that line. This contradicts the
fact that $B$ also contains four points of $\ell$. Hence
$|K|\leq3$, and therefore $|U|\leq6$.

It follows that $q\in\{3,4,5,7,13,16,19\}$. Together with
$q\equiv1\pmod3$, this gives
$q\in\{4,7,13,16,19\}$. Using also $k\mid q^2+q+1$,
$5<k<q^2+q+1$, and $3\mid k$, the only remaining possibilities
are $(q,k)=(16,21)$ and $(16,39)$. In both cases $v=273$.
If $k=21$, then
$\lambda_4=2(v-4)/(k-4)=538/17\notin\mathbb Z$, while if
$k=39$, then $\lambda_4=538/35\notin\mathbb Z$. This
contradicts Lemma~\ref{lem2.1}(iv). Therefore, the case $d=3$
is impossible.

\vspace{2mm}

\textbf{Case 2}: $d>3$.

\vspace{2mm}

Let $H$ be a hyperplane of $\mathbb P^{d-1}(q)$, and choose five
distinct points of $H$. Let $B$ be one of the two blocks
containing them. Let $T(H)$ be the translation group with axis
$H$. Then $T(H)$ fixes $H$ pointwise, acts regularly on
$\mathbb P^{d-1}(q)\setminus H$, and has order $q^{d-1}$.
Writing $K=T(H)_B$, we have $[T(H):K]\leq2$.

Suppose that $B$ contains a point outside $H$. If $q$ is odd,
then $K=T(H)$ and hence $k\geq q^{d-1}+5$. Since
$v=1+q+\cdots+q^{d-1}<2q^{d-1}<2k$, this contradicts
$k\mid v$ and $k<v$.

Suppose that $q$ is even and $q\geq4$. Then
$|K|\geq q^{d-1}/2$, so $k\geq q^{d-1}/2+5$. Moreover,
$v<\frac{q}{q-1}q^{d-1}\leq\frac43q^{d-1}<3k$. Since both
$v$ and $k$ are odd, $v/k$ is an odd integer strictly between
$1$ and $3$, which is impossible.

Finally, let $q=2$. Then $v=2^d-1$ and
$k\geq2^{d-2}+5$. For $d=4,5$, the divisibility
$k\mid2^d-1$ gives no value satisfying $5<k<v$ and this lower
bound. For $d\geq6$, we have $1<v/k<4$. Since $v/k$ is odd,
it follows that $v=3k$. The condition $k-4\mid2(v-4)$ then gives
$k-4\mid2(3k-4)$, and hence $k-4\mid16$. This contradicts
$k-4\geq2^{d-2}+1\geq17$.

Therefore, every block containing five points of $H$ is contained
in $H$. Since $H$ was arbitrary, let $S$ be any $5$-subset and
let $B$ be a block containing $S$. If $\langle S\rangle$ is a
proper subspace, then it is the intersection of all hyperplanes
containing $S$, and hence $B\subseteq\langle S\rangle$. Since
$S\subseteq B$, it follows that
$\langle B\rangle=\langle S\rangle$. If
$\langle S\rangle$ is the whole point space, the same equality is
immediate from $S\subseteq B$. Thus
\(
\langle B\rangle=\langle S\rangle
\)
for every block $B$ containing $S$.

Since $G$ is block-transitive and semilinear transformations
preserve projective dimension, all blocks have the same projective
span dimension. As every $5$-subset is contained in a block, it
follows that all $5$-subsets of the point set must span subspaces
of the same dimension.

This is impossible. If $q\geq4$, five collinear points span a
line, whereas four collinear points together with one point outside
the line span a plane. If $q=2$ or $3$, five suitable points in
a plane span that plane, whereas another five-point set can be
chosen to span a projective $3$-space. Thus $d>3$ cannot occur.

Consequently, the case $N=\P\S\L(d,q)$ with $d\geq3$ is excluded.
\end{proof}

By Lemmas~3.7--3.11, the only possible parameter--group triples
in Case~{\rm (B2)} are
\begin{equation*}
(v,k,G)=(12,6,\P\G\L(2,11))
\quad\text{and}\quad
(v,k,G)=(24,8,\P\G\L(2,23)).
\end{equation*}
Both possibilities occur by Lemma~3.8.

The uniqueness of the first design in Lemma~3.8 follows from
\cite{WeiLi}. We now establish the corresponding
uniqueness statement for the second design.

\begin{lemma}\label{lem3.12}
Let $X=\mathbb{P}^{1}(23)=\GF(23)\cup\{\infty\}$, and $G=\P\G\L(2,23)$ acting naturally on $X$.
There is exactly one $G$-orbit $\mathcal{O}$ on the $8$-subsets of $X$ for which $(X,\mathcal{O})$ is a $5$-$(24,8,2)$ design.
Consequently, there is, up to isomorphism, a unique simple block-transitive $5$-$(24,8,2)$ design whose block set is a single $G$-orbit.
\end{lemma}

\begin{proof}
Let $\mathcal{D}=(X,\mathcal{B})$ be such a design. Since $G$ is
block-transitive, for any $B\in\mathcal{B}$ we have
\begin{equation*}
b=2\frac{\binom{24}{5}}{\binom{8}{5}}=1518,\qquad
|G|=23(23^2-1)=12144,\qquad |G_B|=\frac{|G|}{b}=8.
\end{equation*}
We therefore first determine the $G$-orbits of length $1518$ on
the $8$-subsets of $X$.

The subgroup and overgroup structure of $\P\G\L(2,23)$ follows
from \cite{CameronPGL}, while the
relevant point-orbit lengths follow from
\cite{CameronPGL}. Every subgroup of order $8$
is either cyclic or dihedral. A cyclic subgroup $C_8$ acts
semiregularly on $X$ and hence has three point-orbits of length
$8$. The overgroup data show that each such orbit is stabilized
by a dihedral overgroup of order $16$. Thus no $8$-subset of $X$
has full stabilizer $C_8$.

There are two $G$-conjugacy classes of dihedral subgroups of order
$8$, represented by $H_1$ and $H_2$, where
\begin{equation*}
H_1\leq\P\S\L(2,23)
\quad\text{and}\quad
H_2\not\leq\P\S\L(2,23).
\end{equation*}
Their point-orbit lengths on $X$ are, respectively,
$8, 8, 8$ and $4, 4, 8, 8$.

Consequently, each $H_i$ stabilizes exactly three $8$-subsets of
$X$. Each $H_i$ lies in a unique dihedral subgroup of order $16$,
whose point-orbit lengths are $8$ and $16$. In the case of $H_1$,
there are also overgroups isomorphic to $S_4$, but
\cite[Lemma~10]{CameronPGL} shows that these groups are transitive
on $X$ when $q=23$ and hence stabilize no $8$-subset. It follows
that, for each $i$, exactly one of the three $H_i$-invariant
$8$-subsets has stabilizer properly containing $H_i$. The other
two have full stabilizer $H_i$.

By \cite{CameronPGL}, each of the two conjugacy classes
contains $759$ subgroups. Hence, for each $i$, the number of
$8$-subsets whose full stabilizer is conjugate to $H_i$ is
$759\cdot2=1518$. Since an orbit with stabilizer of order $8$ has
length $12144/8=1518$, the orbit-counting formula in
\cite{CameronPGL} shows that there is exactly one such
$G$-orbit for each of the two conjugacy classes.

Representatives of these two orbits may be chosen as
\begin{equation*}
\begin{split}
B_+&=\{\infty,0,1,3,12,15,21,22\},\\
B_-&=\{\infty,0,1,2,3,6,10,19\},
\end{split}
\end{equation*}
where $G_{B_+}$ is conjugate to $H_1$ and $G_{B_-}$ is conjugate
to $H_2$. Put $\mathcal{O}_+=B_+^G$ and
$\mathcal{O}_-=B_-^G$.

It remains to test the $5$-design condition. Since $G$ is sharply
$3$-transitive on $X$, its orbits on the $4$-subsets of $X$ are
described by the anharmonic action on cross-ratios. The four
anharmonic orbits on $GF(23)\setminus\{0,1\}$ are
\begin{equation*}
\begin{aligned}
A_1&=\{4,6,9,15,18,20\},&
A_2&=\{3,8,11,13,16,21\},\\
A_3&=\{5,7,10,14,17,19\},&
A_4&=\{2,12,22\}.
\end{aligned}
\end{equation*}
Let $\Omega_j$ be the corresponding $G$-orbit on the $4$-subsets
of $X$. For a $5$-subset $S$, define
\begin{equation*}
\rho(S)=\bigl(\rho_1(S),\rho_2(S),\rho_3(S),\rho_4(S)\bigr),
\qquad
\rho_j(S)=\left|\binom{S}{4}\cap\Omega_j\right|.
\end{equation*}
The vector $\rho(S)$ is invariant under $G$.
For each representative $S_i$ below, we classify its five
$4$-subsets by their cross-ratios in $A_1,\ldots,A_4$ to obtain
$\rho(S_i)$. We compute $G_{S_i}$ by testing which elements of
$G$ stabilize $S_i$ setwise, and obtain
$|S_i^G|=|G|/|G_{S_i}|$.
The resulting data, verified using Magma, are listed below.
\begin{center}
\begin{tabular}{c|c|c|c|c}
\hline
$i$ & $S_i$ & $\rho(S_i)$ & $|G_{S_i}|$ & $|S_i^G|$\\
\hline
$1$ & $\{\infty,0,1,3,4\}$ & $(3,2,0,0)$ & $2$ & $6072$\\
$2$ & $\{\infty,0,1,2,5\}$ & $(2,0,2,1)$ & $2$ & $6072$\\
$3$ & $\{\infty,0,1,3,14\}$ & $(2,1,2,0)$ & $2$ & $6072$\\
$4$ & $\{\infty,0,1,2,6\}$ & $(1,1,2,1)$ & $1$ & $12144$\\
$5$ & $\{\infty,0,1,2,3\}$ & $(1,2,0,2)$ & $2$ & $6072$\\
$6$ & $\{\infty,0,1,3,7\}$ & $(0,3,2,0)$ & $2$ & $6072$\\
\hline
\end{tabular}
\end{center}
Indeed, the six profiles are pairwise distinct, and the displayed
orbit lengths sum to
$5\cdot6072+12144=42504=\binom{24}{5}$.
Thus the table accounts for all $G$-orbits on the $5$-subsets of
$X$. Write $\Delta_i=S_i^G$ and set
\begin{equation*}
m_i^\pm=
\left|\binom{B_\pm}{5}\cap\Delta_i\right|.
\end{equation*}
For $S\in\Delta_i$, let $\lambda_i^\pm$ be the number of members
of $\mathcal{O}_\pm$ containing $S$. Double-counting the incident
pairs between $\Delta_i$ and $\mathcal{O}_\pm$ gives
\begin{equation*}
\lambda_i^\pm
=\frac{m_i^\pm|G_{S_i}|}{|G_{B_\pm}|}
=\frac{m_i^\pm|G_{S_i}|}{8}.
\end{equation*}
For each of the $\binom{8}{5}=56$ five-subsets $T$ of $B_\pm$,
we compute $\rho(T)$ and compare it with the six distinct
profiles in the preceding table. Since those profiles distinguish
all six $G$-orbits, $T\in\Delta_i$ if and only if
$\rho(T)=\rho(S_i)$. Counting these occurrences gives
$m_i^\pm$, and the preceding double-counting formula then gives
$\lambda_i^\pm$. The results are listed below.
\begin{center}
\begin{tabular}{c|c|c|c|c|c|c}
\hline
$i$ & $1$ & $2$ & $3$ & $4$ & $5$ & $6$ \\
\hline
$m_i^+$ & $8$ & $8$ & $8$ & $16$ & $8$ & $8$ \\
$m_i^-$ & $0$ & $8$ & $0$ & $32$ & $8$ & $8$ \\
\hline
$\lambda_i^+$ & $2$ & $2$ & $2$ & $2$ & $2$ & $2$ \\
$\lambda_i^-$ & $0$ & $2$ & $0$ & $4$ & $2$ & $2$ \\
\hline
\end{tabular}
\end{center}
Therefore every $5$-subset of $X$ is contained in exactly two
members of $\mathcal{O}_+$, so $(X,\mathcal{O}_+)$ is a simple
$5$-$(24,8,2)$ design. In contrast, $\mathcal{O}_-$ does not
satisfy the $5$-design condition: the members of $\Delta_1$ and
$\Delta_3$ are contained in no block, whereas the members of
$\Delta_4$ are contained in four blocks.

Since $\mathcal{O}_+$ and $\mathcal{O}_-$ are the only
$G$-orbits of the required length, $\mathcal{O}_+$ is the unique
possible block orbit. This proves the lemma.
\end{proof}

\begin{lemma}\label{lem:almost-B3}
Case {\rm(B3)} is impossible.
\end{lemma}

\begin{proof}
In this case,
$|G|=(q^3+1)q^3(q^2-1)a/n$, where $n=(3,q+1)$,$a=|G:N|$ and $a\mid2ne$.
Since $G$ is block-transitive, we have
$b=|G:G_B|=2\binom{v}{5}/\binom{k}{5}$.
After cancelling the common factors, we obtain
\begin{equation}\tag{3.14}\label{eq3.14}
2(q^2+q+1)(q^3-2)(q^3-3)|G_B|
=
\frac{(q+1)a}{n}
k(k-1)(k-2)(k-3)(k-4).
\end{equation}

Since $k\mid q^3+1$, we have
$\gcd(q^2+q+1,(q+1)k)=1$,
$\gcd(q^3-2,(q+1)k)\mid3$ and
$\gcd(q^3-3,(q+1)k)\mid16$.
Indeed, $\gcd(q^2+q+1,q^3+1)\mid2$, while $q^2+q+1$ is odd; moreover, any common divisor of $q^3-2$ and $q^3+1$ divides $3$, and the final bound follows from
$\gcd(q^3-3,q+1)\mid4$ and $\gcd(q^3-3,k)\mid4$.
Thus equation~\eqref{eq3.14}, together with $a\mid2ne$, yields
\begin{equation*}
(q^2+q+1)(q^3-2)(q^3-3)
\mid
96ne(k-1)(k-2)(k-3)(k-4).
\end{equation*}
By Corollary~\ref{cor2.3},
$k<\sqrt{2(q^3-3)}+4$.
Consequently,
\begin{equation*}
(q^2+q+1)(q^3-2)(q^3-3)
<
288e\bigl(\sqrt{2(q^3-3)}+4\bigr)^4.
\end{equation*}
Since $e\leq\log_2q$, this inequality is impossible for $q\geq87$.
It remains only to consider the prime powers $q<87$.

By Lemma~\ref{lem2.18}, we have $\gcd(k,6)>1$. Using this together with $k\mid q^3+1$, $5<k<q^3+1$, $k-4\mid2(q^3-3)$ and $2\binom{q^3+1}{5}/\binom{k}{5}\in\mathbb Z$, we obtain the following remaining possibilities.

\begin{center}
\begin{tabular}{c|c@{\qquad}c|c@{\qquad}c|c}
\hline
$q$ & $k$ & $q$ & $k$ & $q$ & $k$\\
\hline
$5$  & $6$       & $17$ & $6,9$     & $41$ & $6$\\
$7$  & $8$       & $23$ & $6,8,12$  & $47$ & $6,8,9$\\
$9$  & $10$      & $29$ & $6$       & $53$ & $6$\\
$11$ & $6,12$    & $31$ & $8$       & $59$ & $6$\\
$32$ & $9$       & $71$ & $6,8$     & $83$ & $6$\\
$79$ & $8$       &      &           &      &    \\
\hline
\end{tabular}
\end{center}

For each pair $(q,k)$ in the table and each divisor $a$ of $2ne$, equation~\eqref{eq3.14} gives
\begin{equation*}
|G_B|
=
\frac{(q+1)a\,k(k-1)(k-2)(k-3)(k-4)}
{2n(q^2+q+1)(q^3-2)(q^3-3)},
\end{equation*}
which is never an integer.
Hence none of the remaining parameter pairs can occur.
Therefore, the case $N=\P\S\U(3,q)$ is impossible.
\end{proof}

\begin{lemma}\label{lem:almost-B4}
Case {\rm(B4)} is impossible.
\end{lemma}

\begin{proof}
In this case, $v=q^2+1$, where $q=2^{2e+1}$.
Thus $v$ is odd and $v\equiv2\pmod3$.
Since $k\mid v$, we have $2\nmid k$ and $3\nmid k$,
contrary to Lemma~\ref{lem2.18}.
Therefore Case~{\rm(B4)} is impossible.
\end{proof}

\begin{lemma}\label{lem:almost-B5}
Case {\rm(B5)} is impossible.
\end{lemma}

\begin{proof}
In this case, $|G|=(q^3+1)q^3(q-1)a$, where $a\mid2e+1$.
Since $G$ is block-transitive and $b=2\binom{v}{5}/\binom{k}{5}$, cancellation of the common factors gives
\begin{equation}\tag{3.15}\label{eq3.15}
2(q^2+q+1)(q^3-2)(q^3-3)|G_B|
=
a k(k-1)(k-2)(k-3)(k-4).
\end{equation}

Since $k\mid q^3+1$, we have
$\gcd(q^2+q+1,k)=1$,
$\gcd(q^3-2,k)\mid3$ and
$\gcd(q^3-3,k)\mid4$.
Indeed, $\gcd(q^2+q+1,q^3+1)\mid2$, and $q^2+q+1$ is odd.
It follows from equation~\eqref{eq3.15} that
$(q^2+q+1)(q^3-2)(q^3-3)$ divides
$12a(k-1)(k-2)(k-3)(k-4)$.
By Corollary~\ref{cor2.3},$k<\sqrt{2(q^3-3)}+4$.
Since $a\leq2e+1=\log_3q$, we obtain
\begin{equation*}
(q^2+q+1)(q^3-2)(q^3-3)
<
12\log_3q\bigl(\sqrt{2(q^3-3)}+4\bigr)^4.
\end{equation*}

However, $q\geq27$.
For every such $q$, we have
$q^3-2>(26/27)q^3$,
$q^3-3>(26/27)q^3$ and
$\sqrt{2(q^3-3)}+4<(3/2)q^{3/2}$.
Consequently, the left-hand side of the preceding inequality is greater than
$(676/729)q^8$, whereas its right-hand side is less than
$(243/4)q^6\log_3q$.
Thus we would have
$(676/729)q^2<(243/4)\log_3q$.
This inequality already fails for $q=27$, and the quotient
$q^2/\log_3q$ is increasing for $q\geq27$.
Therefore, no admissible value of $q$ exists, and the case
$N=\mathrm{Re}(q)$ is impossible.
\end{proof}

\begin{lemma}\label{lem:almost-B6}
Case {\rm(B6)} is impossible.
\end{lemma}

\begin{proof}
Since $|\Out(N)|=1$ for $N=\Sp(2d,2)$ with $d\geq3$
\cite{KleidmanLiebeck}, we have $G=N=\Sp(2d,2)$.

We first dispose of the cases $3\leq d\leq5$. By
Lemma~\ref{lem2.1}, Corollary~\ref{cor2.3}, and
Lemma~\ref{lem2.18}, the conditions $k\mid v$,
$\lambda_s\in\mathbb Z$ for $0\leq s\leq4$,
$(k-3)(k-4)\leq2(v-4)$, and $\gcd(k,6)>1$ leave only the
following possibilities:
\begin{center}
\begin{tabular}{c|c|c}
\hline
$d$ & $v$ & admissible values of $k$ \\
\hline
$3$ & $28$ & $\varnothing$ \\
$3$ & $36$ & $6$ \\
$4$ & $120$ & $\varnothing$ \\
$4$ & $136$ & $\varnothing$ \\
$5$ & $496$ & $\varnothing$ \\
$5$ & $528$ & $6,8$ \\
\hline
\end{tabular}
\end{center}
For $(v,k)=(36,6)$, Lemma~\ref{lem2.1} gives $b=125664$.
Since $125664\nmid1451520=|\Sp(6,2)|$, this contradicts
Lemma~\ref{lem2.16}.

It remains to consider $v=528$ when $d=5$. For $k=6$ and $8$, the
corresponding numbers of blocks are $111845462960$ and
$11983442460$, respectively, and both are divisible by $131$.
However,
\begin{equation*}
|\Sp(10,2)|
 =2^{25}\prod_{j=1}^{5}(2^{2j}-1)
 =2^{25}\cdot3^6\cdot5^2\cdot7\cdot11\cdot17\cdot31,
\end{equation*}
which is not divisible by $131$. Again this contradicts
Lemma~\ref{lem2.16}. Hence we may assume that $d\geq6$.

Let $(V,\beta)$ be the natural $2d$-dimensional symplectic space
over $\mathbb F_2$. We use the standard quadratic-form model for
the Jordan--Steiner actions of $G$ \cite{Bamberg}.
For $\epsilon\in\{1,-1\}$, let $\mathcal Q^\epsilon$ denote the
set of quadratic forms of sign $\epsilon$ polarizing to $\beta$.
We identify $\mathcal P$ with $\mathcal Q^\epsilon$, where
$|\mathcal Q^\epsilon|=2^{2d-1}+\epsilon2^{d-1}=v$.

Fix $Q\in\mathcal Q^\epsilon$, and, for $a\in V$, define
$Q_a(u)=Q(u)+\beta(u,a)$ for $u\in V$. The map $a\mapsto Q_a$ is
a bijection from $V$ to the set of all quadratic forms polarizing
to $\beta$, and $Q_a\in\mathcal Q^\epsilon$ if and only if
$Q(a)=0$. Thus
$\mathcal Q^\epsilon=\{Q_a\mid Q(a)=0\}$ and $Q_0=Q$.
Moreover, $G_Q\cong\mathrm O^\epsilon(2d,2)$, and its action on
$\mathcal Q^\epsilon$ is equivalent to its natural action on the
$Q$-singular vectors\cite{Bamberg}.

By the Witt decomposition, there exists a non-degenerate
$4$-dimensional subspace $U$ such that $Q_U$ is of minus type
\cite{KleidmanLiebeck}. Put $W=U^\perp$. Then
$V=U\perp W$, and the additivity of the Arf invariant shows that
$Q_W$ has sign $\delta=-\epsilon$. Write $\dim W=2m$, so that
$m=d-2\geq4$.

A $4$-dimensional quadratic space of minus type over $\mathbb F_2$
has six singular vectors, including the zero vector
\cite{Bamberg}. Therefore
\begin{equation*}
\Sigma_U=\{Q_a\mid a\in U,\ Q(a)=0\},
\qquad
S=\Sigma_U\setminus\{Q\}
\end{equation*}
satisfy $|\Sigma_U|=6$ and $|S|=5$. Set
$
H=\mathrm O^-(U,Q_U)\times
  \mathrm O^\delta(W,Q_W)\leq G_Q.
$
Then $H$ stabilizes $S$ setwise. Let $B_1$ and $B_2$ be the two
blocks containing $S$. By Lemma~\ref{lem2.17}, $H$ permutes
$B_1$ and $B_2$, and hence
$[H:H_{B_i}]\leq2$ for $i=1,2$.

We next determine the relevant $H$-orbits. By Witt's lemma,
$\mathrm O^-(U,Q_U)$ is transitive on both the five nonzero
singular vectors and the ten nonsingular vectors of $U$.
Similarly, $\mathrm O^\delta(W,Q_W)$ is transitive on the
nonzero singular vectors and on the nonsingular vectors of $W$,
whose respective numbers are
\begin{equation*}
s_\delta=2^{2m-1}+\delta2^{m-1}-1,
\qquad
n_\delta=2^{2m-1}-\delta2^{m-1}.
\end{equation*}
Every vector outside $U$ has a unique expression $u+w$, where
$u\in U$ and $0\neq w\in W$, and it is $Q$-singular precisely
when $Q_U(u)=Q_W(w)$. It follows that the $H$-orbits on the
$Q$-singular vectors outside $U$ have lengths
$s_\delta$, $5s_\delta$, and $10n_\delta$.

Fix $i\in\{1,2\}$ and put $K=H_{B_i}$. If $K=H$, then every
$H$-orbit remains a single $K$-orbit. Otherwise,
$[H:K]=2$, so $K\lhd H$, and every $H$-orbit is either a single
$K$-orbit or the union of two $K$-orbits of equal length. Since
$s_\delta$ is odd, the $H$-orbits of lengths $s_\delta$ and
$5s_\delta$ cannot split. An $H$-orbit of length $10n_\delta$
either remains a single $K$-orbit or splits into two $K$-orbits
of length $5n_\delta$. Since
\begin{equation*}
5n_\delta-s_\delta
=4\cdot2^{2m-1}-6\delta2^{m-1}+1>0,
\end{equation*}
it follows in either case that every $K$-orbit on
$\mathcal Q^\epsilon\setminus\Sigma_U$ has length at least
$s_\delta$.

Suppose that $B_i$ contains a point outside $\Sigma_U$. Since
$K$ stabilizes $B_i$, the corresponding $K$-orbit is contained
in $B_i\setminus S$. Consequently,
\begin{equation*}
k\geq5+s_\delta
 \geq2^{2d-5}-2^{d-3}+4,
\qquad
k<\sqrt{2v}+4<2^d+5,
\end{equation*}
where the second inequality follows from
Corollary~\ref{cor2.3} and
$2v=2^{2d}+\epsilon2^d<(2^d+1)^2$. For $d\geq6$, however,
\begin{equation*}
2^{2d-5}-2^{d-3}+4-(2^d+5)
 =2^{d-3}(2^{d-2}-9)-1>0,
\end{equation*}
which is a contradiction. Thus $B_i\subseteq\Sigma_U$ for
$i=1,2$.

Finally, $S\subseteq B_i\subseteq\Sigma_U$, where $|S|=5$ and
$|\Sigma_U|=6$. Since the design is non-trivial, $k>5$, and hence
$k=6$ and $B_1=B_2=\Sigma_U$. This contradicts the simplicity of
$\mathcal D$, since the two blocks containing the $5$-subset $S$
must be distinct. Therefore Case~{\rm(B6)} is impossible.
\end{proof}

\begin{lemma}\label{lem:almost-B7}
Case {\rm(B7)} is impossible.
\end{lemma}

\begin{proof}
Suppose first that $N=\P\S\L(2,11)$ and $v=11$.
Since $k\mid11$ and $5<k<11$, there is no admissible value of $k$.

Suppose next that $N=\P\S\L(2,8)$ and $v=28$. The conditions $k\mid28$ and $5<k<28$ give $k\in\{7,14\}$. By Lemma~\ref{lem2.18}, the case $k=7$ is impossible. If $k=14$, then $k-4=10\nmid48=2(v-4)$, contrary to Lemma~\ref{lem2.1}(iv). Therefore, this case cannot occur.
\end{proof}

\begin{lemma}\label{lem:almost-B8}
Case {\rm(B8)} is impossible.
\end{lemma}

\begin{proof}
Since \(N\lhd G\), every \(G\)-orbit on the \(k\)-subsets of
\(\mathcal P\) is a union of \(N\)-orbits. We shall use this
observation below.

If $v=11$ or $23$, the conditions $k\mid v$ and $5<k<v$ admit no possible value of $k$. If $v=22$, then $k=11$, contrary to Lemma~\ref{lem2.18}. Hence these three degrees are excluded.

Suppose that \(v=12\). Then \(k=6\), and Lemma~\ref{lem2.1}(ii)
gives \(b=2\binom{12}{5}/\binom{6}{5}=264.\)
The natural action of \(M_{12}\) on the \(6\)-subsets of its
\(12\)-point set has precisely two orbits, of lengths \(132\) and
\(792\) \cite{BaileyBray}. Since \(\mathcal B\) is a
\(G\)-orbit and \(N=M_{12}\lhd G\), it must be a union of these
\(N\)-orbits. No union of them has size \(264\), contrary to
\(|\mathcal B|=b\).

Finally, suppose that \(v=24\). The conditions \(k\mid24\) and
\(5<k<24\) give \(k\in\{6,8,12\}\). For \(k=12\),
Lemma~\ref{lem2.1}(ii) gives \(b=2\binom{24}{5}/\binom{12}{5}=322/3,\)
which is not an integer. For \(k=6\) and \(8\), the corresponding
numbers of blocks are \(14168\) and \(1518\), respectively.

By Choi's classification of the subset orbits of \(\M_{24}\)
\cite{Choi}, its orbit lengths on \(6\)-subsets are \(21252\) and
\(113344\), while its orbit lengths on \(8\)-subsets are
\(759\), \(97152\), and \(637560\). No union of the \(6\)-subset
orbits has size \(14168\), and no union of the \(8\)-subset orbits
has size \(1518\). Since \(N=M_{24}\lhd G\), neither value can be
the length of a \(G\)-orbit. This contradicts block-transitivity.

Therefore Case~{\rm(B8)} is impossible.
\end{proof}

\begin{lemma}\label{lem:almost-B9}
Case {\rm(B9)} is impossible.
\end{lemma}

\begin{proof}
Suppose first that $N=\M_{11}$ and $v=12$. The conditions $k\mid v$ and $5<k<v$ give $k=6$, and Lemma~\ref{lem2.1}(ii) gives $b={2\binom{12}{5}}/{\binom{6}{5}}=264$.
In its transitive action of degree $12$, the group $\M_{11}$ has three orbits on the $6$-subsets, of lengths $22$, $110$, and $792$ \cite{Crnkovic}. 
Since $N\lhd G$ and $\mathcal B$ is a $G$-orbit, $\mathcal B$ must be a union of these $N$-orbits.
However, $22+110=132<264<792$, so no such union has size $264$.
This is a contradiction.

If $N=\A_7$ and $v=15$, there is no integer $k$ satisfying
$k\mid15$ and $5<k<15$.

Suppose next that $N=\HS$ and $v=176$. Using $k\mid176$,
$5<k<176$, and
$\lambda_4=344/(k-4)\in\mathbb Z$, Lemma~\ref{lem2.1}(iv)
leaves only $k=8$. For this value, however,
\begin{equation*}
\lambda_3
 =2\frac{\binom{173}{2}}{\binom{5}{2}}
 =\frac{14878}{5}\notin\mathbb Z,
\end{equation*}
again contrary to Lemma~\ref{lem2.1}(iv).

Finally, let $N=\Co_3$ and $v=276$. The conditions $k\mid276$,
$5<k<276$, and $\lambda_4=544/(k-4)\in\mathbb Z$ leave
$k\in\{6,12\}$. If $k=12$, then Lemma~\ref{lem2.1}(ii) gives
$b=97491940/3$, which is not an integer. Hence $k=6$, and $
b={2\binom{276}{5}}/{\binom{6}{5}}
  =4289645360.$
By \cite{Atlas}, $\Out(\Co_3)=1$, and hence $G=N=\Co_3$, with\(
|\Co_3|
 =2^{10}\cdot3^7\cdot5^3\cdot7\cdot11\cdot23
 =495766656000.\)
Since $13\mid b$ but $13\nmid|G|$, we have $b\nmid|G|$,
contrary to Lemma~\ref{lem2.16}. Therefore Case~{\rm(B9)} is
impossible.
\end{proof}

The preceding lemmas exclude all cases other than the two
parameter--group triples obtained in Lemma~\ref{lem3.8}.
By \cite[Theorem~1]{WeiLi} and
Lemma~\ref{lem3.12}, respectively, each of these two triples
determines a unique design up to isomorphism.
Consequently, the only non-trivial simple block-transitive
$5$-$(v,k,2)$ designs satisfying $k\mid v$ are the
$5$-$(12,6,2)$ design admitting $\P\G\L(2,11)$ as a
block-transitive automorphism group and the
$5$-$(24,8,2)$ design admitting $\P\G\L(2,23)$ as a
block-transitive automorphism group.
This completes the proof of Theorem~\ref{th1.1}.

\end{document}